\documentclass{amsart}
\usepackage{graphicx}
\usepackage{verbatim}
\usepackage{url}
\newtheorem{thm}{Theorem}[section]
\newtheorem*{acknowledgement*}{Acknowledgement}
\newtheorem{cor}[thm]{Corollary}
\newtheorem{lem}[thm]{Lemma}
\newtheorem{prop}[thm]{Proposition}
\theoremstyle{definition}

\theoremstyle{remark}
\newtheorem{rem}[thm]{Remark}
\numberwithin{equation}{section}

\newcommand{\set}[1]{\left\{#1\right\}}
\newcommand{\Real}{\mathbb R}

\newcommand{\func}[1]{\ensuremath{\mathrm{#1} \:} }

\newcommand{\Div}[0]{\func{div}}

\newcommand{\dist}[0]{\mathrm{dist}}

\newcommand{\xX}[0]{\mathbf{x}}

\newcommand{\LL}[0]{\mathcal{L}}

\title[Asymptotic Structure of eigenfunctions of drift Laplacians]{Asymptotic structure of almost eigenfunctions of drift Laplacians on conical ends}
\author{Jacob Bernstein}
\address{Department of Mathematics, Johns Hopkins University, 3400 N. Charles Street, Baltimore, MD 21218}
\email{bernstein@math.jhu.edu}
\thanks{The author was partially supported by the NSF Grant  DMS-1609340.}

\begin{document}
\begin{abstract}
We use a weighted variant of the frequency functions introduced by Almgren to prove sharp asymptotic estimates for almost eigenfunctions of the drift Laplacian associated to the Gaussian weight on an asymptotically conical end. As a consequence, we obtain a purely elliptic proof of a result of L. Wang on the uniqueness of self-shrinkers of the mean curvature flow asymptotic to a given cone. Another consequence is a unique continuation property for self-expanders of the mean curvature flow that flow from a cone.
\end{abstract}
\maketitle

\section{Introduction}
A hypersurface $\Sigma\subset \Real^{n+1}$ is said to be a \emph{self-shrinker} if it satisfies 
\begin{equation*}
\mathbf{H}+\frac{\xX^\perp}{2}=0
\end{equation*}
where $\mathbf{H}=-H \mathbf{n}$ is the mean curvature vector and  $\xX^\perp$ is the normal component of the position vector. If $\Sigma$ is a self-shrinker, then the family $\set{\sqrt{-t}\Sigma}_{t<0}$ is a solution to the mean curvature flow that moves self-similarly by scaling. We say a self-shrinker, $\Sigma$, is \emph{asymptotically conical} if there is a regular cone $C$ with vertex $0$ so that
$\lim_{\rho\to 0} \rho \Sigma = C$ in $C^\infty_{loc}(\Real^{n+1}\backslash\set{0})$ -- here $\Sigma$ may have compact boundary.

In \cite{WaRigid},  L. Wang proved the remarkable fact that an asymptotically conical self-shrinker, even one with compact boundary, is determined by its asymptotic cone:
\begin{thm}\label{WangThm}
If $\Sigma_1$ and $\Sigma_2$ are two asymptotically conical self-shrinkers with the same asymptotic cone, then there is a $R_U\geq 0$ so $\Sigma_1\backslash B_{R_U}=\Sigma_2\backslash B_{R_U}$.
\end{thm}
To prove Theorem \ref{WangThm}, L. Wang applied a  backwards uniqueness result of Escauriaza, Seregin and  \v{S}ver\'{a}k \cite{ESS} to the difference between the mean curvature flows of the two self-shrinkers.  Roughly speaking, the result of \cite{ESS} says that a solution to the heat equation in the exterior of a ball in $\Real^n$ that does not grow too rapidly spatially (but with no assumptions on the interior boundary) and vanishes at time $0$, must vanish identically for negative times.  This is a type of parabolic unique continuation and is proved using Carleman estimates. See \cite{KotschwarWang,WaRigid2} for related results.

In this article we give a purely elliptic proof of Theorem \ref{WangThm}.  Our argument uses a weighted version of the frequency function introduced by Almgren \cite{Almgrem}. 
In fact, we will provide a sharp description of the asymptotic structure of an almost eigenfunction of the drift Laplacian, $\LL_0=\Delta_g-\frac{r}{2}\partial_r$, on a weakly conical end. 

\begin{thm}\label{MainThm2} If $(\Sigma^n, g, r)$ is a weakly conical end and $u\in C^2(\Sigma)$ satisfies
 $$\left|(\LL_0 +\lambda) u\right|\leq M r^{-2}\left(  |u|+|\nabla_g u|\right)\mbox{ and } \int_{\Sigma} \left(|\nabla_g u|^2+u^2\right) r^{2-4\lambda}e^{-\frac{r^2}{4}}<\infty$$
	then, there are constants $R_0$ and $K_0$, depending on $u$, so that for any $R\geq R_0$
	$$
	 \int_{\set{r\geq R}} \left(u^2+ r^2|\nabla_g u|^2 + r^4\left(\partial_r u-\frac{2\lambda}{r} u\right)^2\right)  r^{-1-n-4\lambda} \leq \frac{K_0}{ R^{n+4\lambda}}\int_{\set{r= R}} u^2.
	$$ 
	Moreover, $u$ is asymptotically homogeneous of degree $2\lambda$ and $
\mathrm{tr}_\infty^{2\lambda} u=a$ for some $a\in H^1(L(\Sigma))$ that satisfies $\alpha^2=\lim_{\rho\to \infty} \rho^{1-n-4\lambda} \int_{\set{r=\rho}} u^2=\int_{L(\Sigma)}a^2$ and
	$$
	\int_{\set{r\geq R}}\left(u^2+r^2\left(|u-A|^2+|\nabla_g u|^2\right) + r^4\left(\partial_r u-\frac{2\lambda}{r} u\right)^2\right)r^{-2-n-4\lambda} \leq \frac{K_0  \alpha^2}{R^2}.
		$$ 
		Here $A\in H^1_{loc}(\Sigma)$ is the leading term of $u$ and $L(\Sigma)$ is the link of the asymptotic cone -- see Appendix \ref{AsymptoticSec}.  	
\end{thm}
\begin{rem} Frequency functions have been used to prove unique continuation results in various elliptic and parabolic settings, e.g., \cite{GL1, GL2, Poon}.  They are most useful in low-regularity contexts and so it is possible that Theorem \ref{MainThm2} could be proved by more classic methods, e.g., \cite[Theorem 17.2.8]{Hormander} and \cite{Mazzeo}.  Note that a novel feature of the present article is that $\LL_0$ has an unbounded coefficient. We refer also to the recent article of Colding-Minicozzi \cite{CM} where a frequency function like the one in this article is considered in a more general setting. 
\end{rem}

A feature of Theorem \ref{MainThm2} is that it also applies, after a simple transformation, to expander type problems. In particular, we prove a sharp unique continuation property for asymptotically conical \emph{self-expanders}. That is, to solutions of
\begin{equation*}
\mathbf{H}-\frac{\xX^\perp}{2}=0.
\end{equation*}
We say a self-expander, $\Sigma$, is asymptotically conical if there is a $C^2$-regular cone with vertex the origin, $C$, so that $\lim_{\rho\to 0} \rho\Sigma = C$ in $C^{2}_{loc}(\Real^{n+1}\backslash\set{0})$.
\begin{thm}\label{ExpanderThm}
Suppose $\Sigma_1$ and $\Sigma_2$ are  asymptotically conical self-expanders.  If 
$$
\lim_{\rho\to \infty} \rho^{n+1} e^{\frac{\rho^2}{4}} \dist_H\left(\Sigma_1\cap \partial B_\rho, \Sigma_2\cap \partial B_{\rho}\right)=0, 
$$
then there is a $R_U\geq 0$ so $\Sigma_1\backslash B_{R_U}=\Sigma_2\backslash B_{R_U}$.  Here $\dist_H$ is Hausdorff distance.
\end{thm} 
This is a mean curvature flow analog of a result of Durelle \cite{Durelle} about Ricci flow self-expanders.  Note that Durelle's proof uses elliptic Carleman estimates.  Finally, we note that two of the estimates of this paper -- Theorems \ref{MainThm2Alt} and \ref{StrongDecayThm} -- are needed in several parts of \cite{BernsteinWangExpander}.

\section{Weakly Conical Ends}
We introduce a weak notion of asymptotically conical end. This will suffice for our purposes and is easy to check in applications.  In this section we fix notation for these objects as well as record some basic computational facts.  See Appendix \ref{AsymptoticSec} for how this definition relates to a more geometric notion of asymptotically conical.

A triple  $(\Sigma, g, r)$ consisting of a smooth $n$-dimensional manifold, $\Sigma$, with $n\geq 2$, a $C^1$-Riemannian metric, $g$ and a proper unbounded $C^2$ function $r: \Sigma\to (R_\Sigma, \infty)$
where $R_\Sigma \geq 1$, is a \emph{weakly conical end} if there is a constant
$\Lambda\geq 0$ so that
$$
||\nabla_g r|-1|\leq \frac{\Lambda}{r^4} \leq \frac{1}{2} \mbox{ and } |\nabla_g^2 r^2 -2g|\leq \frac{\Lambda}{r^2}\leq \frac{1}{2}.
$$
Any Riemannian cone satisfies these conditions with $\Lambda=0$. 
The hypotheses ensure that $r$ has no critical points on $M$ and in fact $|\nabla_g r|\geq \frac{1}{2}$. Hence,  for each $\rho> R_\Sigma$,
$$
S_\rho=r^{-1}(\rho)
$$
is a $C^2$ hypersurface.  As $r$ is proper, $\set{S_\rho}_{\rho>R_\Sigma}$ is a $C^2$ foliation of $\Sigma$ with compact leaves. For any $R\geq R_\Sigma$, let 
$$E_R=\set{p: r(p)>R}$$ and for any $R_2>R_1\geq R_\Sigma$ let 
$$A_{R_2,R_1}=E_{R_1}\backslash \bar{E}_{R_2}=\set{p: R_2>r(p)>R_1}.$$

It will be convenient to use big-O notation.   In particular, we write for functions $f_1,f_2,f_3,h$ and constants $\gamma_1, \ldots, \gamma_N$,
$$
f_1=f_2+f_3 O(h; \gamma_1, \ldots \gamma_N)
$$
if there is a $K=K(n, \Lambda, \gamma_1, \ldots, \gamma_N)>0$ and so that $|f_1-f_2|\leq K |h||f_3|$ on $\Sigma$.  Likewise, if $T_1$ and $T_2$ are tensor fields of the same type we write
$$
T_1=T_2+ O(h; \gamma_1, \ldots, \gamma_N)
$$ if there is a $K=K(n, \Lambda, \gamma_1, \ldots, \gamma_N)>0$ so that $|T_1-T_2| \leq K |h| $ on $\Sigma$.  
Note that all constants are assumed to depend on the dimension of $\Sigma$ and on $\Lambda$.
For example,
$$
|\nabla_g r|=1 +O(r^{-4}) \mbox{ and }\nabla^2_g r^2 =2g +O(r^{-2}).
$$

Consider the following three $C^1$ vector fields on $\Sigma$:
$$ 
\partial_r =\nabla_g r,\; \mathbf{N}=\frac{\nabla_g r}{|\nabla_g r|} \mbox{ and } \mathbf{X} = r\frac{\nabla_g r}{|\nabla_{g} r|^2}.
$$
Observe that all three are normal to each $S_\rho$.  Moreover,  $\mathbf{X}\cdot r= r$, $|\mathbf{N}|=1$ and 
$$
|\partial_r -\mathbf{N}|\leq 2\Lambda r^{-4}, |\partial_r - r^{-1} \mathbf{X}|\leq 6 \Lambda r^{-4} \mbox{ and } |\mathbf{N}-r^{-1} \mathbf{X}|\leq 4 \Lambda r^{-4}.$$
As $\nabla_g^2 r=\frac{1}{2r} (\nabla_g^2 r^2-2 dr \otimes dr)$, the hypotheses ensure that
$$
\nabla_g^2 r= \frac{g-dr \otimes dr}{r}+O(r^{-3}).
$$
Hence, for any $C^1$ vectorfield $\mathbf{Y}$,
$$
\mathbf{Y}|\nabla_g r|^2=2\nabla^2_g r (\mathbf{Y}, \nabla_g r) =\frac{2}{r} g(\mathbf{Y}, \partial_r)-\frac{2}{r}g(\mathbf{Y}, \partial_r) +O(|\mathbf{Y}| r^{-3})=O(|\mathbf{Y}| r^{-3}).
$$
As a consequence, for vector fields $\mathbf{Y}$ and $\mathbf{Z}$,
$$
g(\nabla_\mathbf{Y} \mathbf{X}, \mathbf{Z})=\frac{\nabla^2_g r^2(\mathbf{Y},\mathbf{Z})}{2|\nabla_g r|^2}- g(\mathbf{X},\mathbf{Z})\frac{ \mathbf{Y}\cdot |\nabla_g r|^2}{|\nabla_g r|^2}=g(\mathbf{Y},\mathbf{Z})+O(|\mathbf{Y}||\mathbf{Z}|r^{-2})
$$
and if  $\mathbf{Y}$ and $\mathbf{Z}$ are tangent to $S_\rho$,
$$
g(\nabla_\mathbf{Y} \mathbf{N}, \mathbf{Z})=\frac{\nabla^2_g r(\mathbf{Y},\mathbf{Z})}{|\nabla_g r|} -g(\partial_r, \mathbf{Z})\frac{\mathbf{Y}\cdot |\nabla_g r|}{|\nabla_g r|^2}=\frac{g(\mathbf{Y},\mathbf{Z})}{r}+O(|\mathbf{Y}||\mathbf{Z}|r^{-3}).
$$
Hence, if $A_{S_\rho}$ is the second fundamental form of $S_\rho$ relative to $\mathbf{N}$ and $g_{S_\rho}$ is the induced metric, then for any $\rho>R_\Sigma$,
\begin{align*}
A_{S_\rho}=\rho^{-1} g_{S_\rho}+ O(\rho^{-3}) 
 \end{align*}
and so,  $H_{S_{\rho}}$, the mean curvature of the spheres $S_\rho$ relative to $\mathbf{N}$ satisfies
$$
H_{S_\rho}=\frac{n-1}{\rho}+O(\rho^{-3}).
$$
Finally, the above computations give,
$$
\Div_g \partial_r=\frac{n-1}{r} +O(r^{-3})=\Div_g \mathbf{N} \mbox{ and } \Div_g \mathbf{X}=n +O(r^{-2}).
$$

\section{Basic Estimates in Weighted Spaces}
Fix a weakly conical end $(\Sigma, g, r)$. For any $m\in \Real$ define
$$
\Phi_m: \Real^+\to \Real^+ \mbox{ by } \Phi_m(t)=t^m e^{-\frac{t^2}{4}}.
$$
Abusing notation slightly, consider $\Phi_m:\Sigma\to \Real^+$ given by $\Phi_m(p)=\Phi_m(r(p))$. The triple $(\Sigma, g, \Phi_m)$ is a metric measure space generalizing the end of the Gaussian soliton $\left(\Real^{n}, g_{euc}, e^{-\frac{|\mathbf{x}|^2}{4}}\right)$.  The natural drift Laplacian associated to $(\Sigma, g, \Phi_m)$ is
$$
\mathcal{L}_m=\Delta_g -\frac{r}{2} \partial_r +\frac{m}{r} \partial_r.
$$
This operator will be the primary object of study. In particular, we will focus attention on almost solutions to $\LL_m$.  That is a function, $u\in C^2(\Sigma)$ that satisfy
\begin{equation}\label{LLMainAssump}
|\LL_m u|\leq M r^{-2}\left( |u|+|\nabla_g u|\right).
\end{equation}

 As a first step, we introduce appropriate notions of functions of bounded growth and verify certain properties hold for them. 
For each $R>R_\Sigma$ and $0\leq k\leq 2$, let ${C}^k(\bar{E}_R)$ be the space of $k$-times continuously differentiable functions on $\bar{E}_R$.  On $C^k(\bar{E}_R)$ introduce the following (incomplete) weighted norms
$$
||f||_{m}^2=||f||_{m,0}^2= \int_{\bar{E}_R} f^2 \Phi_m \mbox{ and }
||f||_{m,1}^2=||f||_{m}^2+||\nabla_g f||_{m}^2.
$$
Next consider the following classes of functions in ${C}^k(\bar{E}_R)$ with controlled growth:
$$C^k_m(\bar{E}_R)=\set{f\in {C}^k(\bar{E}_R): ||f||_m<\infty} \mbox{ and}$$
$$C^k_{m,1}(\bar{E}_R)=\set{f\in {C}^k(\bar{E}_R):||f||_{m,1}<\infty}.$$
In what follows,  fix an $R>R_\Sigma$ and an element $u\in C^{2}(\bar{E}_{R})$. For each $\rho\geq R$, let 
$$
B(\rho)=\int_{S_\rho} u^2 |\nabla_g r| \mbox{ and } \hat{B}_m(\rho)= \Phi_m(\rho) B(\rho)
$$
be a boundary $L^2$ norm (suitably adapted to the geometry) along with its $\Phi_m$-weighted variant (here we view $S_\rho=\partial E_\rho$).  Likewise, let
$$
F(\rho)=\int_{\partial E_\rho} u (- \mathbf{N}\cdot u) =-\int_{S_\rho} \frac{u \partial_r u}{|\nabla_g r|} \mbox{ and }\hat{F}_m(\rho)=\Phi_m(\rho) F(\rho),
$$
be the unweighted and $\Phi_m$-weighted boundary energy flux, and, when $B(\rho)>0$,
$$
{N}(\rho)= \frac{\rho {F}(\rho)}{{B}(\rho)}= \frac{\rho \hat{F}_m(\rho)}{\hat{B}_m(\rho)}
$$
be the corresponding frequency functions.
If $u\in C^{2}_{m,1}(\bar{E}_R)$, then let
$$\hat{D}_m(\rho)=\int_{\bar{E}_\rho }|\nabla_g u|^2 \Phi_m \mbox{ and } D_m(\rho) = \Phi_m(\rho)^{-1} \hat{D}_m(\rho)$$
be the $\Phi_m$-weighted Dirichlet energy and a normalized version of it
and 
$$
\hat{N}_m(\rho)=\frac{\rho \hat{D}_m(\rho)}{\hat{B}_m(\rho)}  =\frac{\rho D_m(\rho)}{{B}(\rho)} 
$$
be the corresponding frequency functions. 
It is also convenient to set 
$$
\hat{L}_m(\rho) = \int_{E_\rho} u (\LL_m u) \Phi_m \mbox{ and } L_m(\rho)=\Phi_m(\rho)^{-1} \hat{L}_m(\rho).
$$
so that if $\LL_m u\in C^0_m(\bar{E}_R)$, then integration by parts is justified and
$$
\hat{F}_m(\rho)=\hat{D}_m(\rho)+\hat{L}_m(\rho)\mbox{ and } F(\rho) ={D}_m(\rho)+{L}_m(\rho) .
$$

The following Poincar\'{e} inequality holds for elements of $C^2_{m,1}(\bar{E}_R)$.
\begin{prop}\label{GradControlsL2Prop}\label{EasyCor}
	There is an $R_P=R_P( \Lambda, n, m)$ so that if $R\geq R_P$ and $u\in C^{2}_{m,1}(\bar{E}_R)$, then
	$$
   \int_{\bar{E}_R} u^2 \Phi_m \leq  \frac{32}{R^2}\hat{D}_m(R)+\frac{16}{R}  \hat{B}_m(R).
	$$
\end{prop}
\begin{proof}
Consider the vector field $\mathbf{V}(p)=r(p)u^2(p) \Phi_m(p) \mathbf{N}$
and compute
\begin{align*}
\Div_g \mathbf{V}&= r u^2 \Phi_m \Div_g (\mathbf{N} )+ u^2 \Phi_m |\nabla_g r| + 2 r u \Phi_m g(\nabla_g u, \mathbf{N})  + ru^2 g(\nabla_g \Phi_m, \mathbf{N})\\
 &=n u^2 \Phi_m+ 2 r u \Phi_m g(\nabla_g u, \mathbf{N})  - \frac{r^2}{2}  u^2 \Phi_m+m u^2 \Phi_m +O( r^{-2}; m) u^2 \Phi_m. 
\end{align*}
The absorbing inequality and the fact that $|\mathbf{N}|=1$ imply that
$$\Div_g \mathbf{V}\leq (n+m)u^2 \Phi_m- \frac{r^2}{4}  u^2 \Phi_m+4|\nabla_g u|^2\Phi_m+u^2 \Phi_mO(r^{-2}; m). 
$$
By picking $R_P\geq 4(n+1+|m|)^{1/2}$ large enough, if $R\geq R_P$, then
$$\Div_g \mathbf{V}\leq 4|\nabla_g u|^2\Phi_m- \frac{r^2}{8}  u^2\Phi_m.
$$
As $u\in C^2_{m,1}(\bar{E}_R)$, the divergence theorem can be used to prove the claim.  
\end{proof}

Finally, we record some facts that will be used in the sequel. First, observe that simple integration by parts gives, for all $\rho\geq 1$,
\begin{align}\label{PhimIntEqn}
\begin{split}
\int_\rho^{\rho+1}\Phi_m(t) dt& =\frac{2}{\rho}\Phi_m(\rho)-\frac{2}{\rho+1}\Phi_m(\rho+1)+2(m-1)\int_{ \rho}^{\rho+1} \Phi_{m-2}(t) dt \\
&= \frac{2}{\rho}\Phi_m(\rho)+\Phi_m(\rho)O(\rho^{-2}; m).
\end{split}
\end{align}
\begin{lem} \label{LHatPrimeLem}
	We have
	$$
	B'(\rho)=\frac{n-1}{\rho} {B}(\rho) - 2{F}(\rho) +B(\rho)O( \rho^{-3})
	$$
	and
	$$
	\hat{B}'_m(\rho)=\frac{n+m-1}{\rho} \hat{B}_m(\rho) -\frac{\rho}{2} \hat{B}_m(\rho) - 2\hat{F}_m(\rho) + \hat{B}_m(\rho)O( \rho^{-3}).
	$$
\end{lem}
\begin{proof}
	The first variation formula and the weak conical hypotheses ensure
	\begin{align*}
	 B'(\rho) &= \int_{S_\rho} \left( u^2|\nabla_g r| (r^{-1} g(\mathbf{X},\mathbf{N}) H_{S_\rho}) +2 u (r^{-1}\mathbf{X}\cdot u)|\nabla_g r| +u^2 r^{-1}\mathbf{X}\cdot |\nabla_g r|\right)\\
		  &=\frac{n-1}{\rho}  B(\rho)  -2F(\rho)+  B(\rho)O( \rho^{-3})
   \end{align*}
   which is the first identity.  The second is a straightforward consequence of this.
 \end{proof}

\begin{lem}\label{SmallL1CompareLem}  There is an $R_2=R_2(n, \Lambda, m, M)$ and a $K_2=K_2(M)\geq 0$ so that if $u\in C^2_{m,1}(\bar{E}_R)$ satisfies \eqref{LLMainAssump}
	then for all $\rho\geq R_1$,
	\begin{align} \label{FirstSmallL1Eqn}
	\begin{split}|\hat{L}_m(\rho)|&\leq\frac{1}{8}\rho^{-2} \left(\hat{D}_m(\rho)+K_2\rho^{-1} \hat{B}_m(\rho)\right)\leq \frac{1}{4}\rho^{-2} \left(\hat{F}_m(\rho)+ K_2\rho^{-1}\hat{B}_m(\rho)\right) 
\end{split}\end{align}
and
\begin{align}
\label{SecondSmallL1Eqn}
\begin{split}
\left|\int_{\bar{E}_\rho} (\mathbf{X}\cdot u )\right. & \left.(\LL_m u) \Phi_m\right| \leq \frac{K_2}{\rho}\left(\int_{ \bar{E}_\rho} (\mathbf{N}\cdot u)^2 \Phi_m\right)^{\frac{1}{2}}\left( \hat{D}_m(\rho)+\frac{1}{\rho}\hat{B}_{m}(\rho)\right)^{\frac{1}{2}}\\
&\leq \frac{K_2}{\rho}\left( \hat{D}_m(\rho) +\frac{1}{\sqrt{\rho}} \hat{D}_m^{\frac{1}{2}}(\rho)\hat{B}_m^{\frac{1}{2}}(\rho)\right) \leq\frac{2K_2}{\rho} \left(\hat{D}_m(\rho)+\frac{1}{\rho}\hat{B}_m(\rho) \right).
\end{split}
\end{align}
\end{lem}
\begin{proof}
	The absorbing inequality and \eqref{LLMainAssump} imply that
	\begin{align*}
	\left|\hat{L}_m(\rho) \right| &\leq  \int_{\bar{E}_\rho} |u|\left|\LL_m u\right| \Phi_m \leq  
	M\int_{{E}_\rho} r^{-2} |u| \left(|u|+|\nabla_g u|\right) \Phi_m\\
	&\leq   \frac{1}{16} \rho^{-2}\int_{{E}_\rho}\left( |\nabla_g u|^2+ 16(M+4M^2)|u|^2\right) \Phi_m 
	\end{align*}
Hence, $\LL_m u\in C^0_m(\bar{E}_\rho)$ and the first inequality of \eqref{FirstSmallL1Eqn} follows from Proposition \ref{EasyCor} for large enough $R_2\geq R_P+1$ and $K_2$.  The second inequality follows immediately.

Similarly, the Cauchy-Schwarz inequality gives
\begin{align*}
\left|\int_{\bar{E}_\rho} (\mathbf{X}\cdot u ) (\LL_m u) \Phi_m\right|& \leq 4  \int_{\bar{E}_\rho} r |\mathbf{N}\cdot u| \left|\LL_m u\right| \Phi_m \\
&\leq 4 M \int_{\bar{E}_\rho} r^{-1} |\mathbf{N}\cdot u|\left(|\nabla_g u|+ |u|\right) \Phi_m\\
&\leq \frac{4M}{\rho}\left(\int_{ \bar{E}_\rho} (\mathbf{N}\cdot u)^2 \Phi_m\right)^{\frac{1}{2}}\left( \hat{D}_m(\rho)+\int_{\bar {E}_\rho} u^2 \Phi_m\right)^{\frac{1}{2}}\\
\end{align*}
Hence, as $\rho\geq R_2\geq 1$, Proposition \ref{EasyCor} implies that
$$
\left|\int_{\bar{E}_\rho} (\mathbf{X}\cdot u ) (\LL_m u) \Phi_m\right| \leq \frac{4M}{\rho}\left(\int_{ \bar{E}_\rho} (\mathbf{N}\cdot u)^2 \Phi_m\right)^{\frac{1}{2}}\left( 36 \hat{D}_m(\rho)+\frac{36}{\rho}\hat{B}_m(\rho)\right)^{\frac{1}{2}}.
$$
 Increasing $K_2$, if needed, so $K_2\geq 24M$ gives the first inequality of \eqref{SecondSmallL1Eqn}. The second follows from $|\mathbf{N}\cdot u|\leq |\nabla_g u|$ and the third from the absorbing inequality.
\end{proof}

\begin{lem}\label{Harnacklem} There is a $R_3=R_3( n ,\Lambda)$ so that, for any $R\geq R_3$, if $u\in C^2(\bar{A}_{R+2 R})$ satisfies $-1\leq N(\rho)\leq N_+$,
	for  $\rho\in [R,R+2]$ and $N_+\geq 0$, then, for all $\tau\in [0,2]$,
	$$
	\left(1-2N_+\frac{\tau}{R}\right) B(R)\leq B(R+\tau)\leq 	\left(1+2(n+3)\frac{\tau}{R}\right) B(R).
	$$  
\end{lem}
\begin{proof}
	For all $\rho\in [R,R+2]$, Lemma \ref{LHatPrimeLem} gives
	$$
\frac{d}{d\rho} \log \left(B(\rho)\right)=\frac{n-1}{\rho} -2\frac{F(\rho)}{B(\rho)} +O(\rho^{-3}).
	$$
	As $n\geq 2$, for $R_3$ sufficiently large this yields,
	$$
	\frac{1}{2}\geq \frac{n+3}{R}\geq \frac{d}{d\rho} \log \left(B(\rho)\right)\geq -2N_+ R^{-1} .
	$$
	The inequality follows from integrating this and using that $\exp(x)\geq 1+x$ always and $\exp(x)\leq 1+2x$ when $x\in [0,1]$.
\end{proof} 

\begin{lem} \label{TrivialLem}If $u\in C^{2}_{m, 1}(E_R)$ satisfies \eqref{LLMainAssump} for $R\geq R_2$, then either $u$ is trivial or $B(\rho)>0$ for all $\rho>R$.
\end{lem}
\begin{proof}
Suppose that $B(\rho_0)=0$. This means $u$ vanishes on $S_{\rho_0}$ so $\hat{F}_m(\rho_0)=0$.
As $u\in C^{2}_{m, 1}(E_R)$, \eqref{LLMainAssump} implies $\LL_m u\in C^0_m(E_R).$ Hence, as $\rho_0\geq R_2$,  by Lemma \ref{SmallL1CompareLem},
$$
\hat{D}_m(\rho_0)=\hat{F}_m(\rho_0)-\hat{L}_m(\rho_0)\leq \frac{1}{8}\rho_0^{-2} \hat{D}_m(\rho_0)\leq \frac{1}{2}\hat{D}_m(\rho_0).
$$
As $\hat{D}_m(\rho_0)\geq 0$, $\hat{D}_m(\rho_0)=0$ and so $u$ identically vanishes on $\bar{E}_{\rho_0}$. The claim follows from standard unique continuation results (e.g., \cite[Theorem 1.1]{GL2}).
\end{proof}

\section{Frequency Bounds for almost $\mathcal{L}_m$ harmonic functions}
The main result of this section is that, under the assumption \eqref{LLMainAssump}, the two notions of frequency tend to $0$ asymptotically.
\begin{thm}\label{FrequencyThm}
If $u\in C^2_{m+2,1}(\bar{E}_R)$ satisfies \eqref{LLMainAssump} and is non-trivial,
then 
$$
\lim_{\rho\to \infty} \hat{N}_m(\rho)=\lim_{\rho\to \infty} N(\rho)=0.
$$
\end{thm}
When $M=\Lambda=0$, an involved, but straightforward computation, shows that $\hat{N}_m$ is non-increasing in $\rho$ and so the limit exists.  The Gaussian weight, $\Phi_m$, contributes a term with a good sign that, in fact,  forces $\hat{N}_m$ to $0$. When $M>0$, certain technical difficulties appear which are resolved using ideas from \cite{GL2}.

We first record the computation of the change of weighted Dirichlet energy.
\begin{prop} \label{DlambdaPrimeProp}
If $u\in C^{2}_{m+2,1}(\bar{E}_{R}), \LL_m u\in C^0_{m}(\bar{E}_R)$ and $\rho\geq R$, then
\begin{align*}
\hat{D}_m'(\rho) 
&= -\frac{2}{\rho} \int_{ E_\rho}   (\mathbf{X}\cdot u)  (\mathcal{L}_m u) \Phi_m - 2 \int_{S_\rho} \frac{(\mathbf{N}\cdot u)^2}{|\nabla_g r|}\Phi_m+\left( \frac{n+m-2}{\rho}-\frac{\rho}{2}\right) \hat{D}_m(\rho)\\
& -\frac{1}{\rho}\int_\rho^\infty  t \hat{D}_m(t) dt +  O(\rho^{-3})\hat{D}_m(\rho) .
\end{align*}
\end{prop}
\begin{proof}
The co-area formula and divergence theorem give,
\begin{align*}
\hat{D}_m'(\rho) &= - \Phi_m(\rho) \int_{S_\rho} \frac{|\nabla_g u|^2}{|\nabla_g r|} = \rho^{-1} \int_{S_\rho} |\nabla_g u|^2 \Phi_m(r) \mathbf{X} \cdot(- \mathbf{N})\\
   &=\frac{1}{\rho} \int_{E_\rho} \Div_g\left( |\nabla_g u|^2 \Phi_m(r)\mathbf{X} \right) \\
   &=\frac{1}{\rho}\int_{ E_\rho} \left (2 \nabla^2_g u (\mathbf{X} , \nabla_g u) +(n+m-\frac{r^2}{2}+O(r^{-2}))|\nabla_g u|^2 \right) \Phi_m\\
&= \frac{1}{\rho} \int_{ E_\rho} \left (2 \nabla^2_g u (\mathbf{X} , \nabla_g u) +\left(n+m -\frac{r^2}{2}\right)|\nabla_g u|^2 \right) \Phi_m+ O(\rho^{-3}) \hat{D}_m(\rho).
\end{align*}
Combining this with a weighted version the Rellich-Ne\v{c}as identity,
\begin{align*}
 \nabla^2_g u (\mathbf{X}, \nabla_g u)\Phi_m &=
\Div_g \left(   (\mathbf{X}\cdot u)\nabla_g u \Phi_m \right) -  |\nabla_g u|^2\Phi_m \\
&- \mathbf{X}\cdot u\LL_m u\Phi_m+|\nabla_g u|^2\Phi_m O(r^{-2}),
\end{align*}
 and another application of the divergence theorem gives,
\begin{align*}
\hat{D}_m'(\rho)&=   \rho^{-1} \int_{ E_\rho} \left((-2 \mathbf{X}\cdot u)  (\mathcal{L}_m u)+(n+m-2) |\nabla_g u|^2 -\frac{r^2}{2}|\nabla_g u|^2 \right) \Phi_m \\
&-2 \rho^{-1} \int_{S_\rho} (\mathbf{X}\cdot u)(\mathbf{N}\cdot u)\Phi_m+O(\rho^{-3})\hat{D}_m(\rho)\\
&=   \rho^{-1} \int_{ E_\rho} \left((-2\mathbf{X}\cdot u)  (\mathcal{L}_m u) +\left(\frac{\rho^2}{2}-\frac{r^2}{2}\right)|\nabla_g u|^2 \right) \Phi_m\\
&+\left(\frac{n+m-2}{\rho}-\frac{\rho}{2}\right) \hat{D}_m(\rho) 
-2\  \int_{S_\rho}\frac{(\mathbf{N}\cdot u)^2}{|\nabla_g r|}\Phi_m+O(\rho^{-3} )\hat{D}_m(\rho).
\end{align*}
Strictly speaking, to justify the two applications of the divergence theorem, one should work on $\bar{A}_{\tau, \rho}$ and use $u\in C^{2}_{m+2,1}(\bar{E}_{R})$ and $\LL_m u\in C^0_{m}(\bar{E}_R)$ to take $\tau\to \infty$. 

Finally, as  $u\in C^2_{m+2}(\bar{E}_R)$, the co-area formula and Fubini's theorem give,
\begin{align*}
\int_{\rho}^{\infty} &t \hat{D}_m(t) dt =\int_{\rho}^{\infty} t\int_{E_t} |\nabla_g u|^2  \Phi_m dt =\int_{\rho}^{\infty} t  \int_t^\infty \Phi_m(\tau) \int_{S_\tau} \frac{|\nabla_g u|^2}{|\nabla_g r|} d\tau dt\\
 &=\int_{\rho}^{\infty}\left( \int_{\rho}^\tau t dt \right)\Phi_m(\tau) \int_{S_\tau} \frac{|\nabla_g u|^2}{|\nabla_g r|}  d\tau=\int_{\rho}^\infty \frac{1}{2}\left( \tau^2-\rho^2\right) \Phi_m(\tau)\int_{S_\tau}\frac{|\nabla_g u|^2}{|\nabla_g r|}  dt\\
 &=\int_{E_\rho} \left( \frac{r^2}{2}-\frac{\rho^2}{2}\right)|\nabla_g u|^2\Phi_m. \\
\end{align*}

Combining this with the previous computation proves the results.
\end{proof}

\begin{cor}\label{NLambdaEst2cor}
If $u\in C^2_{m+2,1}(\bar{E}_R)$, $\LL_m u \in C^0_m(\bar{E}_R)$ and $B(\rho)>0$ at $\rho\geq R$, then,  
\begin{align*}
\hat{N}_m'(\rho) 
&=\frac{ -2\int_{ \bar{E}_\rho}  \left( \mathbf{X}\cdot u +N(\rho) u\right) (\mathcal{L}_m u) \Phi_m -\int_{\rho}^\infty t \hat{D}_m(t) dt}{\hat{B}_m(\rho)}\\
&-\frac{\frac{2}{\rho}\int_{S_\rho} \left( \mathbf{X}\cdot u +N(\rho)u \right)^2 |\nabla_g r| \Phi_m}{\hat{B}_m(\rho)}+  \hat{N}_{m}(\rho)O(\rho^{-3})
\end{align*}
\end{cor}
\begin{proof}
By Proposition \ref{DlambdaPrimeProp} and Lemma \ref{LHatPrimeLem},
\begin{align*}
\hat{B}_m^2(\rho) \hat{N}_m'(\rho) &= \hat{B}_m(\rho)\hat{D}_m(\rho) +\rho \hat{B}_m(\rho) 
\hat{D}_m'(\rho) - \rho \hat{D}_m(\rho)\hat{B}'_m(\rho) \\
&= -2 \hat{B}_m(\rho)\int_{ \bar{E}_\rho} (\mathbf{X}\cdot u)  (\mathcal{L}_m u ) \Phi_m - 2\rho \Phi_m(\rho) \hat{B}_m(\rho) \int_{S_\rho}\frac{(\mathbf{N}\cdot u)^2}{|\nabla_g r|}   \\
&+2\rho\hat{F}_m(\rho)\hat{D}_m(\rho) -\hat{B}_m(\rho)\int_{\rho}^{\infty} t \hat{D}_m(t) dt +\hat{B}_m(\rho)\hat{D}_m(\rho)O(\rho^{-2}).  
\end{align*}	
As $\LL_m u\in C^0_m(\bar{E}_R)$,  $2\rho\hat{F}_m(\rho)\hat{D}_m(\rho)=2\rho\hat{F}_m^2(\rho)-2\rho \hat{F}_m(\rho) \hat{L}_m(\rho)$ and so
\begin{align*}
 -2\rho  \hat{B}_m(\rho) \int_{S_\rho} \frac{(\mathbf{N}\cdot u)^2}{|\nabla_g r|} \Phi_m+2\rho\hat{F}_m^2(\rho) &=-\frac{2}{\rho} \hat{B}_m(\rho)\int_{S_\rho} \left(\mathbf{X}\cdot u +N(\rho)u \right)^2 |\nabla_g r| \Phi_m.
\end{align*}
Hence,
\begin{align*}
\hat{B}_m^2(\rho) \hat{N}_m'(\rho)&= -2 \hat{B}_m(\rho)\int_{ \bar{E}_\rho} (\mathbf{X}\cdot u+N(\rho)u)  (\mathcal{L}_m u ) \Phi_m -\hat{B}_m(\rho)\int_{\rho}^{\infty} t \hat{D}_m(t) dt\\
&-\frac{2}{\rho} \hat{B}_m(\rho)\int_{S_\rho} \left(\mathbf{X}\cdot u +N(\rho)u \right)^2 |\nabla_g r| \Phi_m+\hat{B}_m(\rho)\hat{D}_m(\rho)O(\rho^{-2}).  
\end{align*}
Dividing by $\hat{B}_m^2(\rho)$  gives the claimed equality.
\end{proof}
If $M=0$, then $\LL_m u=0$ and Corollary \ref{NLambdaEst2cor} says $\hat{N}_m(\rho)$ is almost non-increasing. However, when $M>0$ there is a term $N(\rho)\hat{L}_m(\rho)$ that is hard to control.   Using ideas of Garofalo-Lin \cite{GL2} -- see also \cite{Survey} -- we show that $N(\rho)$ grows at most linearly. This is enough to show $\hat{N}_m(\rho)$ is almost non-increasing.

\begin{prop}\label{NHatPrimeProp}
Suppose $u\in C^{2}_{m+2, 1}(\bar{E}_R)$ is non-trivial and satisfies \eqref{LLMainAssump}
then, there is an $R_4=R_4(M, m, n,\Lambda)$
so that if $\rho\geq R_4$, then
\begin{align*}
N'(\rho)\leq  \frac{1}{4}\rho^{-1} |N(\rho)|+ O(1; M, m).
\end{align*}
\end{prop}

\begin{proof}
As $u$ satisfies \eqref{LLMainAssump} and  $u\in C^{2}_{m+2, 1}(\bar{E}_R)$, $\LL_m u\in C^0_m(\bar{E}_R)$. Hence,
$\hat{F}_m(\rho)=\hat{D}_m(\rho)+\hat{L}_m(\rho)$, and so the co-area formula gives
\begin{equation}\label{FPrimeEqn}\hat{F}'_m(\rho) = \hat{D}_m'(\rho)+\int_{S_\rho}\frac{ u(\LL_m u)}{|\nabla_g r|} \Phi_m=\int_{S_\rho} \frac{|\nabla_g u|^2}{|\nabla_g r} \Phi_m+\int_{S_\rho}\frac{ u(\LL_m u)}{|\nabla_g r|} \Phi_m.
\end{equation}
Applying the Cauchy-Schwarz inequality and \eqref{LLMainAssump} gives,
$$
\left|\int_{S_\rho} \frac{u(\LL_m u)}{|\nabla_g r|}\Phi_m\right|\leq  2M \rho^{-2}\left( \hat{B}_m(\rho)+\hat{B}_m^{1/2}(\rho)\left(\int_{S_\rho} |\nabla_g u|^2\Phi_m \right)^{1/2}\right).
$$
Take $R_4\geq R_2\geq 1$, Proposition \ref{DlambdaPrimeProp} and Lemma \ref{SmallL1CompareLem}, give for $\rho\geq R_4\geq 1$,
\begin{align*}
\hat{D}_m'(\rho) &\leq  - 2 \int_{S_\rho}\frac{(\mathbf{N}\cdot u)^2}{|\nabla_g r|}\Phi_m +\left( \frac{n+m-2}{\rho}-\frac{\rho}{2}\right) \hat{D}_m(\rho) \\
&+\left( \hat{D}_{m}(\rho)+ \rho^{-1}\hat{B}_m(\rho)\right) O(\rho^{-2}; M)\\
&= - 2 \int_{S_\rho}\frac{(\mathbf{N}\cdot u)^2}{|\nabla_g r|}\Phi_m +\left( \frac{n+m-2}{\rho}-\frac{\rho}{2}\right)\hat{F}_m(\rho)\\
&+\frac{1}{8\rho} \left(\hat{F}_m(\rho)+\frac{K_2}{\rho} \hat{B}_m(\rho) \right) +\left( \hat{F}_m(\rho)+\frac{1}{\rho}\hat{B}_m(\rho)\right) O(\rho^{-2}; m, M).
\end{align*}
Hence, for $\rho\geq R_4$,
\begin{align*}
\hat{F}'_m(\rho) &\leq  -2 \int_{S_\rho}\frac{(\mathbf{N}\cdot u)^2}{|\nabla_g r|}\Phi_m +\left(\frac{n+m-2}{\rho}-\frac{\rho}{2} \right)\hat{F}_m(\rho)+\frac{1}{8} \rho^{-1} \hat{F}_m(\rho)\\
&+2 M \rho^{-2}\left( \hat{B}_m(\rho)\int_{S_\rho} |\nabla_g u|^2 \Phi_m \right)^{1/2}+\left( \hat{F}_m(\rho)+ \hat{B}_m(\rho)\right)O( \rho^{-2}; m,M).
\end{align*}

To prove the main claim, first suppose that at $\rho$ the following holds:
\begin{equation}
\label{GradFreqAssumpEqn}
\hat{B}_m(\rho) \int_{S_\rho} |\nabla_g u|^2\Phi_{m} \leq 4\left( \hat{F}_m(\rho) +\frac{\rho}{4}\hat{B}_m(\rho)\right)^2.
\end{equation}
 Lemma \ref{LHatPrimeLem} and Lemma \ref{TrivialLem}, the estimate on $\hat{F}_m'$ and Cauchy-Schwarz give
\begin{align*}
\hat{B}_m^2(\rho) N'(\rho) &=\rho\hat{F}'_m(\rho) \hat{B}_m(\rho)-\rho\hat{F}_m(\rho)
 \hat{B}'_m(\rho)+\hat{F}_m(\rho)\hat{B}_m(\rho)\\
& \leq \frac{1}{8} \rho^{-1} N(\rho) \hat{B}_m^2(\rho)+\frac{2M}{\rho}\hat{B}_m^{3/2} (\rho)\left(\int_{S_\rho} |\nabla_g u|^2 \Phi_m \right)^{1/2} \\
&+ \hat{B}_m^2(\rho) \left( \rho^{-1} N(\rho)+1 \right) O(\rho^{-1}; m,M)\\
&\leq\frac{1}{8} \rho^{-1} |N(\rho)| \hat{B}_m^2(\rho)+\hat{B}_m^2(\rho) \left( \rho^{-2} |N(\rho)|+1 \right) O(1; m,M).
\end{align*}
The result follows immediately from this for $R_4$ sufficiently large.

If \eqref{GradFreqAssumpEqn} does not hold, then \eqref{FPrimeEqn} and the absorbing inequality give,
\begin{align*}
\hat{B}_m&(\rho)\hat{F}'_m(\rho)= -\hat{B}_m(\rho) \int_{S_\rho}\frac{|\nabla_g u|^2}{|\nabla_g r|}\Phi_m - \hat{B}_m(\rho)\int_{S_\rho}\frac{u\LL_m u}{|\nabla_g u|} \Phi_m \\
&\leq  -\frac{2}{3}\hat{B}_m(\rho) \int_{S_\rho}|\nabla_g u|^2\Phi_m+ \frac{1}{6} \hat{B}_m(\rho) \int_{S_\rho} |\nabla_g u|^2 \Phi_m+12\frac{M+M^2}{\rho^2} \hat{B}_m^2(\rho)\\
&\leq  -2 \left(\hat{F}_m(\rho) +\frac{\rho}{4} \hat{B}_m(\rho)\right)^2+12\frac{M+M^2}{\rho^2} \hat{B}_m^2(\rho)\leq -2 \hat{F}_m^2(\rho)-\rho \hat{F}_m(\rho) \hat{B}_m(\rho),
\end{align*} 
where the last inequality requires $R_4^4\geq 96 (M+M^2)$. Hence,
\begin{align*}
\hat{B}_m^2(\rho) N'(\rho) &=\rho \hat{B}_m(\rho)\hat{F}'_m(\rho)-\rho\hat{F}_m(\rho)
 \hat{B}'_m(\rho)+\hat{F}_m(\rho)\hat{B}_m(\rho)\\
 &\leq  -2\rho \hat{F}_m^2(\rho) -\rho^2 \hat{B}_m(\rho)\hat{F}_m(\rho)  +(2-n-m+\frac{\rho^2}{2}) \hat{F}_m(\rho)\hat{B}_m(\rho) \\ &
 +2 \rho\hat{F}_m^2(\rho)+ \hat{F}_m(\rho)\hat{B}_m(\rho) O(\rho^{-2})\\
 &\leq (2-n-m-\frac{1}{2}\rho^2) \hat{F}_m(\rho)\hat{B}_m(\rho)+ \hat{B}_m(\rho) |\hat{F}_m(\rho)| O(\rho^{-2}).\\
\end{align*}
By Lemma \ref{SmallL1CompareLem} and the non-negativity of $\hat{D}_m(\rho)$, as $\rho\geq R_4\geq R_2$, 
$$\hat{F}_m(\rho)\geq -\frac{1}{3} K_3 \rho^{-3}\hat{B}_m(\rho).$$
 Hence, up to increasing $R_4$ so $R_4^2>2(2-n-m)+6$, 
$$
\hat{B}_m^2(\rho) N'(\rho)\leq K_3 \rho^{-1} \hat{B}_m^2(\rho)+ \hat{B}_m(\rho) |\hat{F}_m(\rho)| O(\rho^{-2})
$$
The desired result follows from this, possibly after further increasing $R_4$.
\end{proof}
\begin{cor}\label{NLinGrowthCor}
Suppose $u\in C^{2}_{m+2,1}(\bar{E}_R)$ is non-trivial and satisfies \eqref{LLMainAssump}.
There is an $R_5\geq R$, depending on $u$, 
so that if $\rho\geq R_5$, then
\begin{align*}
\hat{N}_m'(\rho)&\leq -\frac{\int_{\rho}^\infty t \hat{D}_m(t) dt}{\hat{B}_m(\rho)}+(\hat{N}_m^{1/2}(\rho)+\hat{N}_m(\rho)+\rho^{-3})O(\rho^{-2}; m, M).\\
\end{align*}
%
\end{cor}
\begin{proof}
We first use Proposition \ref{NHatPrimeProp} to establish the linear growth of $|N(\rho)|$. To that end let $I$ be any component of $\set{\rho\geq R_4: N(\rho)>0}$.  At the left endpoint of $I$, $\rho_I$, either $N(\rho_I)=N(R_4)\geq 0$ or $N(\rho_I)=0$.  By Proposition \ref{NHatPrimeProp}, for $t\in I$, $
N'(t)\leq \frac{1}{4} t^{-1} N(t)+\frac{1}{2}K_4$
for some $K_4=K_4(n,\Lambda, m, M)$. That is, for any $t\in I$
$$\left(\frac{N(t)}{t^{1/4}}-\frac{2}{3}K_4t^{3/4}\right)'\leq 0 \mbox{ and so } N(t)\leq \frac{2}{3}K_4 t + t^{1/4} \frac{N(\rho_I)}{\rho_I^{1/4}}.$$
Hence, by picking $R_5\geq R_4$ sufficiently large, depending on $N(\rho_I)$ and hence on $u$, for all $t\in I$, $t\geq R_5$, $N(t)\leq K_4 t$.  As the bound is independent of the interval $I$, for all $\rho\geq R_5$, $N(\rho)\leq K_4 \rho$.
Finally, by Lemma \ref{SmallL1CompareLem}, $N(\rho)\geq -\frac{1}{3} K_3 \rho^{-2}$ and so up to further increasing $R_4$, $|N(\rho)|\leq  K_4\rho$.

Corollary \ref{NLambdaEst2cor} gives the following estimate for $\rho\geq R_5$,
\begin{align*}
\hat{N}_m'(\rho)&\leq 2 \frac{ \left|N(\rho)\hat{L}_m(\rho)+\int_{\bar{E}_\rho} (\mathbf{X}\cdot u)(\LL_m u) \Phi_m\right|}{\hat{B}_m(\rho)}-\frac{\int_{\rho}^\infty t \hat{D}_m(t) dt}{\hat{B}_m(\rho)} +\hat{N}_m(\rho)O(\rho^{-3}).
\end{align*}
Hence, Lemma \ref{SmallL1CompareLem} and the linear growth estimate give for $\rho\geq R_5\geq R_2$,
\begin{align*}
\hat{N}_m'(\rho)&\leq   K_4 \rho^{-2} \hat{N}_m(\rho) +\frac{K_2}{4}\rho^{-3}|N(\rho)| +2 K_2\rho^{-2} \left(\hat{N}_m^{\frac{1}{2}}(\rho)+\hat{N}_m(\rho)\right)\\
&-\frac{\int_{\rho}^\infty t \hat{D}_m(t) dt}{\hat{B}_m(\rho)} +\hat{N}_m(\rho)O(\rho^{-3}).  
\end{align*}
 Lemma \ref{SmallL1CompareLem} implies $|N(\rho)|\leq 2 \hat{N}_m(\rho)+K_2 \rho^{-2}$, which gives the desired bound.
\end{proof}

\begin{proof}(of Theorem \ref{FrequencyThm})
First observe, that
$$
\lim_{\rho\to \infty} \hat{N}_m(\rho)=N_0\in [0, \infty)
$$
exists. 
Indeed, by Corollary \ref{NLinGrowthCor} and the absorbing inequality, 
there is a constant $K>0$ so that for $\rho\geq R_5$
$$
\hat{N}_m'(\rho)\leq K\rho^{-2}\left( \hat{N}_m(\rho)+1\right) \mbox{ and so }\frac{d}{d\rho}\left(e^{\frac{K}{\rho}} (\hat{N}_m(\rho)+1)\right)\leq 0.
$$
That is, $e^{\frac{K}{\rho}} (\hat{N}_m(\rho)+1)$ is a non-increasing function and so the limit exists.  As $\hat{N}_m(\rho)\geq 0$ by definition, $N_0\in [0, \infty)$.  By Lemma \ref{SmallL1CompareLem}, $\lim_{\rho\to \infty} N(\rho)=N_0$.  In what follows, we assume $N_0>0$ and aim to derive a contradiction.

As $N_0>0$, there is a $\rho_{0}\geq R_3+34N_0$ so that for $\rho\geq\rho_0$,  $\frac{1}{2} N_0\leq \hat{N}_m(\rho)\leq 2N_0$ and $\frac{1}{2} N_0\leq N(\rho)\leq 2 N_0$.
By Lemma \ref{Harnacklem}, for all $\rho\geq \rho_0\geq R_3$ and $\tau\in [0, 2]$, 
$$
\frac{3}{4}B(\rho)\leq \left( 1- \frac{8N_0}{\rho}\right) B(\rho)\leq B(\rho+\tau).
$$
Hence, for $\rho_0$ sufficiently large, \eqref{PhimIntEqn} gives
\begin{align*}
\int_{\rho}^\infty &t \hat{D}_m(t) dt\geq  \int_{\rho}^{\rho+1} t \hat{D}_m(t)  dt = \int_{\rho}^{\rho+1} \hat{N}_m(t) \hat{B}_m(t) dt \\
&\geq \frac{3}{8}N_0 B(\rho) \int_{\rho}^{\rho+1} \Phi_m (t) dt 
\geq  \frac{1}{4}N_0  \rho^{-1} \Phi_m(\rho) B(\rho)= \frac{1}{4} N_0 \rho^{-1} \hat{B}_m(\rho).
\end{align*}
And so, by Corollary \ref{NLinGrowthCor}, $ \hat{N}_m'(\rho)\leq - \frac{1}{4}N_0 \rho^{-1}+O(\rho^{-2}; m,M, N_0).$
The right hand side is not integrable, a contradiction that proves the claim.
\end{proof}

\section{Frequency Decay for almost $\mathcal{L}_m$ harmonic functions}
In this section we show that for functions $u$ that satisfy \eqref{LLMainAssump} not only do the frequencies decay to zero at infinity, but they do so at a definite (and sharp) rate. 

\begin{thm}\label{FrequencyThm2}
If $u\in C^2_{m+2,1}(\bar{E}_R)$ satisfies \eqref{LLMainAssump} and is non-trivial,
then 
$$
\lim_{\rho\to \infty} \rho^2 \hat{N}_m(\rho)= \xi[u]\in [0, \infty).
$$
In particular, there is a $\rho_{-1}\geq R$ so that for $\rho\geq \rho_{-1}$ and $\bar{\xi}=\max\set{2\xi[u],1}$,
$$
\hat{N}_m(\rho)\leq  \rho^{-2}\bar{\xi}\leq 1 \mbox{ and }(\xi[u]-K_2)\rho^{-2} \leq N(\rho)\leq (\xi[u]+K_2)\rho^{-2}.
$$
\end{thm}
\begin{rem}
Assumption \eqref{LLMainAssump} is not enough to ensure
$\lim_{\rho \to \infty}\rho^2 N(\rho) $
exists. 
\end{rem}
In order to prove this the term $-\int_{ \rho}^\infty t \hat{D}_m(t) dt$ is again exploited. As before, $M>0$ introduces certain technical issues related to error terms in the differential inequality for $\hat{N}_m$.  To overcome these issues we will iterate on the decay rate.
Specifically, we will consider $u\in C^2_{m+2,1}(\bar{E}_R)$ that satisfy for all $\rho\geq R$
	\begin{equation}
	\label{NGrowthDecayAssumpEqn}
	\hat{N}_m(\rho)\leq \eta \rho^{2\gamma} \leq \eta
	\end{equation}
  where $\gamma\in [-1,0]$ and $\eta>0$.
A consequence of Theorem \ref{FrequencyThm} is that \eqref{NGrowthDecayAssumpEqn} holds for $\gamma=0$ and some $\eta>0$. The iteration will ultimate improve the decay to $\gamma=-1$.

To close the argument, an improvement on the error term in Corollary \ref{NLinGrowthCor} is needed. The following will guarantee this when $\hat{N}_m$ doesn't drop too fast.
\begin{lem}\label{ImprovedRadialL2Lem}
	If $u\in C^2_{m+2,1}(\bar{E}_R)$ is non-trivial and satisfies \eqref{LLMainAssump} and \eqref{NGrowthDecayAssumpEqn} for some $\gamma$ and $\eta$, then there are constants, $R_6=R_6(n,\Lambda,\eta, m,M )\geq R$ and $K_6=K_6(n,\Lambda,\eta, m,M )$, so that, for $\rho\geq R_6$,
	$$
	\int_{E_{\rho}}  (\mathbf{N}\cdot u)^2 \Phi_m \leq \frac{K_6}{\rho}\left(\hat{N}_m(\rho)-\hat{N}_m(\rho+2)\right)\hat{B}_m(\rho) +{K}_6\rho^{-2+2\gamma} \hat{B}_m(\rho).
	$$
\end{lem}
\begin{proof}
	Let $\tilde{D}(\rho)=\rho^{-n}{D}_m(\rho)$.  By Proposition \ref{DlambdaPrimeProp}, for $R_6$ sufficiently large,
	\begin{align*}
	\tilde{D}'(\rho)&=-2 \rho^{-n} \int_{S_\rho} \frac{(\mathbf{N}\cdot u)^2}{|\nabla_g r|}-2 \rho^{-1-n} \Phi_m^{-1}(\rho) \int_{ E_\rho}   (\mathbf{X}\cdot u)  (\mathcal{L}_mu) \Phi_m-\frac{2}{\rho} \tilde{D}(\rho) \\
	&-\rho^{-1-n} \Phi_m^{-1}(\rho)\int_{\rho}^\infty t \hat{D}_m(t) dt + \tilde{D}(\rho) O(\rho^{-3})\\
	&\leq - \rho^{-n} \int_{S_\rho} (\mathbf{N}\cdot u)^2-2 \rho^{-1-n} \Phi_m^{-1}(\rho)\int_{ E_\rho}   (\mathbf{X}\cdot u)  (\mathcal{L}_mu) \Phi_m-\frac{1}{\rho} \tilde{D}(\rho).
	\end{align*}
   For $R_6\geq R_2$ sufficiently large,  Lemma \ref{SmallL1CompareLem} and \eqref{NGrowthDecayAssumpEqn} give,
	\begin{align*}
	&\tilde{D}'(\rho) \leq -\frac{1}{\rho^n} \int_{S_\rho} (\mathbf{N}\cdot u)^2+\frac{2K_2}{\rho^{2+n}\Phi_m(\rho)}\left( \hat{D}_m(\rho)+\frac{1}{\rho^{\frac{1}{2}}}\hat{D}_m^{\frac{1}{2}}(\rho)\hat{B}_{m}^{\frac{1}{2}}(\rho)\right) -\frac{1}{\rho} \tilde{D}(\rho)\\
	& \leq- \frac{1}{\rho^n} \int_{S_\rho} (\mathbf{N}\cdot u)^2+ \frac{2K_2}{ \rho^{3+n}}\hat{N}_m^{\frac{1}{2}}(\rho) {B}(\rho) \leq  -\frac{1}{ \rho^{n}} \int_{S_\rho} (\mathbf{N}\cdot u)^2+ \frac{B(\rho)}{\rho^n} O(\rho^{-3+\gamma}; M, \eta).
	\end{align*} 
	Integrating and using Lemma \ref{Harnacklem} and the co-area formula gives, 
	$$
	\tilde{D}(\rho+2)-\tilde{D}(\rho)\leq -\frac{1}{(\rho+2)^n}\int_{\bar{A}_{\rho+2,\rho}} |\nabla_g r| (\mathbf{N}\cdot u)^2+\frac{1}{\rho^n}B(\rho) O(\rho^{-3+\gamma}; M, \eta).
	$$
	Hence, using $\gamma\geq -1$, $\frac{1}{2}<|\nabla_g r|$, $(\rho+2)\leq 2^n\rho^n$, Lemma \ref{Harnacklem} and \eqref{NGrowthDecayAssumpEqn},
	\begin{align*}
	\frac{1}{2^{n+1}}\rho^{-n}&\int_{\bar{A}_{\rho+2,\rho}} (\mathbf{N}\cdot u)^2\leq  \tilde{D}(\rho)-\tilde{D}(\rho+2)+ B(\rho) O(\rho^{-2-n+2\gamma}; M, \eta)\\
	&= \frac{\hat{N}_m(\rho) B(\rho)}{\rho^{n+1}} -\frac{\hat{N}_{m}(\rho+2) B(\rho+2)}{(\rho+2)^{n+1}} + B(\rho) O(\rho^{-2-n+2\gamma}; M, \eta)\\
	&\leq   \frac{1}{\rho^{n+1}}\left(\hat{N}_m(\rho)-  \hat{N}_{m}(\rho+2)\right)B(\rho)+B(\rho) O(\rho^{-2-n+2\gamma};M,\eta)
	\end{align*}	
By making $R_6$ sufficiently large, for all $t\in [\rho, \rho+2]$,
	$$
	\frac{\Phi_m(t)}{\Phi_m(\rho)}\leq 2\mbox{ and so }
	\int_{\bar{A}_{\rho+2,\rho}}(\mathbf{N}\cdot u)^2\Phi_m \leq 2\Phi_m(\rho)\int_{\bar{A}_{\rho+2,\rho}} (\mathbf{N}\cdot u)^2.  
	$$
	By Lemma \ref{Harnacklem} and \eqref{NGrowthDecayAssumpEqn} and the fact that $\Phi_{m}^{-1}(\rho)\Phi_{m}(\rho+2)\leq O(e^{-\rho}; m)$,
	$$
	\int_{\bar{E}_{\rho+2}}  (\mathbf{N}\cdot u)^2\Phi_m\leq \hat{D}_m(\rho+2) = \frac{\hat{N}_m(\rho+2)}{\rho+2} \hat{B}_m(\rho+2)\leq \hat{B}_m(\rho)O(e^{-\rho}; m, \eta).
	$$
	Adding these two estimates, and possibly increasing $R_6$, proves the results.	
\end{proof}

\begin{prop}\label{GrowthorDecayProp}
If $u\in C^2_{m+2,1}(\bar{E}_R)$ is non-trivial and satisfies \eqref{LLMainAssump} and \eqref{NGrowthDecayAssumpEqn} for some $\gamma$ and $\eta$, then there are constants
$R_7\geq R$ and $\Gamma\geq 0$, depending on $u$, so that, for $\rho\geq R_7$, either
\begin{enumerate}
\item $\hat{N}_m(\rho+2)-\hat{N}_m(\rho)\leq - 4\rho^{-1}\hat{N}_m(\rho)$, or
\item $\hat{N}_m(\rho+1)-\hat{N}_m(\rho)\leq - 2\rho^{-1}\hat{N}_m(\rho)+\Gamma \rho^{-2+2\gamma}.$
\end{enumerate}
\end{prop}
\begin{proof}
By Corollary \ref{NLinGrowthCor}, Theorem \ref{FrequencyThm} and $\gamma\in [-1,0]$, there are an $R_7\geq R$ and $\kappa\geq 0$ so that when $s\geq R_7$, \eqref{NGrowthDecayAssumpEqn} ensures that
\begin{equation*}
\hat{N}_m'(s)\leq -\frac{\int_{ s}^\infty t \hat{D}_m(t) dt}{\hat{B}_m(s)} + \frac{1}{10}\kappa s^{-2+\gamma}\leq \frac{1}{10}\kappa s^{-2+\gamma}.
\end{equation*}
As $\gamma\leq 0$, integrating gives, for any $\tau\in [0,2]$ and $s\geq R_7$,
\begin{equation}\label{NhatGrowthEst}
\hat{N}_m(s+\tau)\leq  \hat{N}_m(s) +\frac{1}{5}\kappa s^{-2+\gamma}.
\end{equation}
Suppose Case (1) does not hold for a given $\rho$, that is
$$
\hat{N}_m(\rho+2)-\hat{N}_m(\rho)> -4 \rho^{-1} \hat{N}_m (\rho) 
$$
we first claim that,
for all $\tau\in [0,2]$,
\begin{equation}
\label{Case1FailEqn}
\hat{N}_m(\rho+\tau)-\hat{N}_m(\rho)\geq -4 \rho^{-1} \hat{N}_m (\rho)-\kappa \rho^{-2+\gamma}.
\end{equation}
Indeed, if \eqref{Case1FailEqn} fails for some $\tau_0\in[0,2]$, then, by \eqref{NhatGrowthEst} applied at $s=\rho+\tau_0$,
\begin{align*}
\hat{N}_m&(\rho+2) =\hat{N}_m(\rho+\tau_0+(2-\tau_0))\leq \hat{N}_m(\rho+\tau_0)  +\frac{\kappa}{5} (\rho+\tau_0)^{-2+\gamma}\\
&\leq \hat{N}_m(\rho+\tau_0)+ \frac{\kappa}{5} \rho^{-2+\gamma}<\hat{N}_m(\rho) -4\rho^{-1}\hat{N}_m(\rho) - \frac{4}{5}\kappa \rho^{-2+\gamma}\\
&<\hat{N}_m(\rho)-4\rho^{-1} \hat{N}_m(\rho).
\end{align*}
This contradicts the assumption that Case (1) doesn't hold at $\rho$, proving the claim.

Take $R_7\geq R_1$. By Lemma \ref{Harnacklem} and $\hat{N}_m(\rho)\leq \eta$, for all $\rho\geq R_7$ and $t\in [\rho, \rho+2]$, $ \left( 1- 2\eta\rho^{-1}\right) B(\rho)\leq B(t)$,
Thus, by \eqref{Case1FailEqn},
\begin{align*}
t&{D}_m(t)\geq \left( 1-  \frac{2\eta}{\rho}\right)\hat{N}_m(t) B(\rho)\geq   \left( 1-  \frac{2\eta}{\rho}\right)\left(\left( 1-   \frac{4}{\rho}\right)  \hat{N}_m(\rho)-\kappa \rho^{-2+\gamma}\right)B(\rho)\\
&\geq (1-(2\eta+4)\rho^{-1})\hat{N}_m(\rho)B(\rho)-\kappa \rho^{-2+\gamma} B(\rho).
\end{align*}
Hence, using \eqref{PhimIntEqn}, \eqref{NGrowthDecayAssumpEqn} and $\gamma\geq -1$ and  taking $\Gamma=\Gamma(m,\eta,\kappa)$ sufficiently large
\begin{align*}
\int_{\rho}^\infty& t \hat{D}_m(t) dt  \geq \int_{\rho}^{\rho+1} t D_m(t) \Phi_m(t) dt \geq \left(1-\frac{2\eta+4}{\rho}\right)\hat{N}_m(\rho)B(\rho)\int_{\rho}^{\rho+1} \Phi_m (t) dt\\
& -\kappa \rho^{-1+2\gamma} B(\rho)\int_{\rho}^{\rho+1} \Phi_m (t) dt \geq  \frac{2}{\rho}\hat{N}_m(\rho) \hat{B}_m(\rho) -\frac{1}{10}\Gamma \rho^{-2+2\gamma}\hat{B}_m(\rho).
\end{align*}
By Lemma \ref{ImprovedRadialL2Lem} and the fact that Case (1) doesn't hold but \eqref{NGrowthDecayAssumpEqn} does,
\begin{align*}
	\int_{\bar{E}_\rho}&  (\mathbf{N}\cdot u)^2  \Phi_m  \leq K_6\rho^{-1}\left(\hat{N}_m(\rho)-\hat{N}_m(\rho+2)\right)\hat{B}_m(\rho) +{K}_6\rho^{-2+2\gamma} \hat{B}_m(\rho)\\
	&\leq 4K_6 \rho^{-2}{\hat{N}_m(\rho)}\hat{B}_m(\rho)+{K}_6\rho^{-2+2\gamma} \hat{B}_m(\rho)\leq \hat{B}_m(\rho) O(\rho^{-2+2\gamma}; {K}_6, \eta).
\end{align*}
This estimate, \eqref{SmallL1CompareLem} and Corollary \ref{NLambdaEst2cor} yield, after possibly further increasing $\Gamma$,
$$
\hat{N}_m'(\rho)\leq -2 \rho^{-1} \hat{N}_m(\rho) +\frac{2}{10} \Gamma \rho^{-2+2\gamma}.
$$
Hence, using \eqref{Case1FailEqn}, for any $\tau\in [0,1]$,
\begin{align*}
\hat{N}_m'&(\rho+\tau)\leq -2 (\rho+\tau)^{-1} \hat{N}_m(\rho+\tau) +\frac{2}{10} \Gamma (\rho+\tau)^{-2+2\gamma}\\
 &\leq -2 \rho^{-1} \hat{N}_m(\rho+\tau)+\frac{3}{10}  \Gamma\rho^{-2+2\gamma}\leq  -2 \rho^{-1} \hat{N}_m(\rho)+8 \rho^{-2} \hat{N}_m(\rho)+\frac{3}{10}  \Gamma\rho^{-2+2\gamma}\\
 & \leq  -2 \rho^{-1} \hat{N}_m(\rho)+  \Gamma\rho^{-2+2\gamma}\\
\end{align*}
Where, up to increasing $\Gamma$, the last inequality follows from \eqref{NGrowthDecayAssumpEqn}.
Integrating, $\tau$ over $[0,1]$ show that Case (2) holds at $\rho$ if Case (1) does not.
\end{proof}
\begin{proof} (of Theorem \ref{FrequencyThm2}).
Set $\Xi(\rho)=\rho^2 \hat{N}_m(\rho)$ and suppose \eqref{NGrowthDecayAssumpEqn} holds for some $\gamma$ and $\eta$. If Case (1) of  Proposition \ref{GrowthorDecayProp} holds at $\rho$, then
\begin{align*}
\Xi(\rho+2)&-\Xi(\rho)= (\rho+2)^2\hat{N}_m(\rho+2) -\rho^2 \hat{N}_m(\rho)\\
 &\leq -4\rho \hat{N}_m(\rho) + 4 \rho \hat{N}_m(\rho+2)+4 \hat{N}_m(\rho+2)\\
 &\leq -16 \hat{N}_m(\rho)+ 4\hat{N}_m(\rho)-16 \rho^{-1} \hat{N}_m(\rho)\leq 0.
\end{align*}
If Case (2) of  Proposition \ref{GrowthorDecayProp} holds at $\rho\geq 1$, then
\begin{align*}
\Xi(\rho&+1)-\Xi(\rho) = (\rho+1)^2 \hat{N}_m(\rho+1) -\rho^2  \hat{N}_m(\rho)\\
&\leq -2\rho \hat{N}_m(\rho)+\Gamma\rho^{2\gamma} + 2\rho \hat{N}_m(\rho+1) + \hat{N}_m(\rho+1)\\
&\leq -4 \hat{N}_m(\rho)+ \hat{N}_m(\rho)-2\rho^{-1} \hat{N}_m(\rho)+\Gamma\rho^{2\gamma} +2\Gamma\rho^{-1+2\gamma}+\Gamma \rho^{-2+2\gamma}\leq 4\Gamma\rho^{2\gamma}.
\end{align*}

For each $i\geq 0$, let
$$\min_{\rho\in [R+i, R+i+2]} \Xi(\rho)=\Xi^-_i\leq \Xi^+_i=\max_{\rho\in [R+i, R+i+2]} \Xi(\rho).$$ 
By Theorem \ref{FrequencyThm}, there is an $\eta_0>0$, so that \eqref{NGrowthDecayAssumpEqn}  holds with $\gamma=0$ and $\eta=\eta_0$. 
Hence, if either case of Proposition \ref{GrowthorDecayProp} holds, then
$$
\Xi_{i+1}^+\leq \Xi_i^+ +4\Gamma \mbox{ and so } \Xi_i^+\leq \Xi_0^+ +4\Gamma i.
$$
That is,  if $\rho\in [R+i, R+i+2]$, then
$$
\rho^2 \hat{N}_m(\rho)=\Xi(\rho)\leq \Xi_{i}^+ \leq  \Xi_0^+ +4\Gamma i\leq \Xi_0^+ +4\Gamma (\rho-R)\leq  \Xi_0^++4\Gamma \rho.
$$
Hence, \eqref{NGrowthDecayAssumpEqn} holds for $\gamma=-1/2$ and $\eta= \Xi_0^++4\Gamma$.

Plugging the improved decays bounds into Proposition \ref{GrowthorDecayProp} gives a new $\Gamma'$ for which we repeat the above arguments to obtain,
$$
\Xi_{i+1}^+\leq \Xi_i^++ \frac{4\Gamma'}{R+i}\mbox{ and so } \Xi_i^+\leq  \Xi_0^++4 \Gamma'\sum_{j=1}^i \frac{1}{j}\leq \Xi_0^+ + 4\Gamma'(1+\ln i).
$$
As such, \eqref{NGrowthDecayAssumpEqn} holds for $\gamma=-3/4$ and $\eta$ sufficiently large.
Hence,
$$
\Xi_{i+1}^+\leq \Xi_i^++ \frac{4\Gamma''}{(R+i)^{\frac{3}{2}}}\mbox{ and so }\Xi_i^+\leq \Xi_0^++ 4\Gamma'' \sum_{j=1}^i \frac{1}{j^{\frac{3}{2}}}\leq \Xi_0^++12\Gamma''.
$$
That is, \eqref{NGrowthDecayAssumpEqn} holds for $\gamma=-1$ and $\eta= \Xi_0^++12\Gamma''.$

It remains to show that $\lim_{\rho\to \infty} \Xi(\rho)$ exists.  First observe that the above estimates imply that $\xi_i=\Xi_i^+-4\Gamma'' \sum_{j=1}^{i-1} j^{-\frac{3}{2}} $ is monotone non-increasing and  uniformly bounded from below. Hence, $\lim_{i\to \infty} \xi_i$ exists and is finite and so
$$
\xi=\lim_{i\to \infty} \Xi_i^+ =\lim_{i\to \infty}( \xi_i+4\Gamma'' \sum_{j=1}^{i-1} j^{-\frac{3}{2}})
$$
is finite.  Let $\rho_i^-\in[R+i,R+i+2]$ satisfy $\Xi(\rho_i^-)=\Xi_i^-$.
As \eqref{NGrowthDecayAssumpEqn} holds with $\gamma=-1$, \eqref{NhatGrowthEst} implies that, for $i$ large, there is a $\kappa>0$ so that, for $\rho\in [R+i+2, R+i+4]$,
$$
 \hat{N}_m(\rho) \leq \hat{N}_m(R+i+2) + \frac{1}{2}\kappa i^{-3}\leq  \hat{N}_m(\rho_i^-)+\kappa i^{-3}.
$$
Hence, $\Xi_{i+2}^+-K (i+1)^{-1}\leq \Xi_i^-\leq \Xi_i^+$ for sufficiently large $K$ and so
$$
\lim_{i\to \infty} \Xi_i^-=\lim_{i\to \infty} \Xi_i^+=\xi \mbox{ and hence }\lim_{\rho\to \infty} \Xi(\rho)=\xi\in [0, \infty)
$$
which proves the main claim.  The bounds on $N(\rho)$ follow from Lemma \ref{SmallL1CompareLem}.
\end{proof}

\section{Asymptotic Estimates at infinity}
In this section we will use the previous results about the frequency in order to describe the asymptotic behavior of reasonable functions satisfying \eqref{LLMainAssump}.  We refer to Section \ref{AsymptoticSec} for some of the definitions. This proves Theorem \ref{MainThm2} when $\lambda=0$.

\begin{thm}\label{AsymptoticEstThm}
 If $u\in C^2_{m+2,1}(\Sigma)$ satisfies \eqref{LLMainAssump}
	then, there are constants $R_8$ and $K_8$, depending on $u$, so that for $R\geq R_8$
	$$
	 \int_{\bar{E}_R} \left(u^2+  r^2|\nabla_g u|^2 + r^{4}(\partial_r u)^2\right)r^{-1-n} \leq \frac{K_8}{R^n}\int_{S_R} u^2.
	$$ 
	Moreover, $u$ is asymptotically homogeneous of degree $0$ and there is an element $a\in H^1(L(\Sigma))$ satisfying $
	\mathrm{tr}_\infty^0 u=a$ so that $\alpha^2=\lim_{\rho\to \infty} \rho^{1-n} \int_{S_{\rho}} u^2=\int_{L(\Sigma)}a^2$ and
		$$
		\int_{\bar{E}_R}\left(u^2+r^2\left(|u-A|^2+|\nabla_g u|^2\right) + r^4\left(\partial_r u\right)^2\right)r^{-2-n} \leq \frac{ K_8 \alpha^2}{R^2}.
			$$ 
Here $A\in H^1_{loc}(\Sigma)$ is the leading term of $u$ and $L(\Sigma)$ is the link of the asymptotic cone.  	
\end{thm}
\begin{proof}
By Lemma \ref{LHatPrimeLem} and  Theorem \ref{FrequencyThm2}, if $R_8\geq \rho_{-1}$, then, for all $\rho\ge R_8$,
$$
\frac{d}{d\rho} \left(\rho^{1-n} B(\rho)\right) =O(\rho^{-3};\xi, K_2)\rho^{1-n}B(\rho).
$$
As the right hand side is integrable,
$$
\lim_{\rho\to \infty} \rho^{1-n} B(\rho)=\alpha^2\in [0, \infty)
$$
exists. Hence, for $R_8$ sufficiently large, if $\rho\geq R_8$, then
$$
\frac{1}{2}\alpha^2 \leq \rho^{1-n} B(\rho)\leq 2\alpha^2.
$$
By the co-area formula, for any $R\geq R_8$,
$$
\int_{E_R} r^{-1-n} u^2 \leq 2\int_{R}^\infty t^{-1-n} B(t) dt \leq 4\alpha^2  \int_{R}^\infty  t^{-2}dt = 4\alpha^2 R^{-1}\leq 8 R^{-n} B(R).
$$
By Theorem \ref{FrequencyThm2}, as $\rho\geq R_8\geq \rho_{-1}$, 
$$
{D}_m (\rho) =\rho^{-1} \hat{N}_m(\rho) {B}(\rho) \leq \bar{\xi} \rho^{-3} B(\rho)\leq 2\bar{\xi}  \alpha^2 \rho^{n-4}. 
$$

Next, observe that by the co-area formula,
$$
\frac{d}{d\rho}\left( \rho^{1-n} D_m(\rho)\right)\leq \left(\frac{\rho}{2}-\frac{m}{\rho} \right) \rho^{1-n} {D}_m(\rho)-\frac{1}{2}\int_{S_{\rho}} \rho^{1-n} |\nabla_g u|^2
$$
Hence, as long as $R_8\geq 4|m|$ is large enough, when $\rho\geq R_8$, 
$$
\frac{1}{2}\int_{S_{\rho}} \rho^{1-n} |\nabla_g u|^2\leq  2\bar{\xi} \alpha^2\rho^{-2} - \frac{d}{d\rho}\left( \rho^{1-n} D_m(\rho)\right).
$$
Integrating, using the co-area formula and the decay of $D_m$ gives, for $R\geq R_8$,
$$
\int_{\bar{E}_R} r^{1-n} |\nabla_g u|^2\leq \frac{8\bar{\xi}  \alpha^2}{R}+4 \frac{R D_m(R)}{R^n}\leq \frac{16\bar{\xi}  \alpha^2}{R}\leq   \frac{32\bar{\xi}B(R)}{R^n}. 
$$
By Proposition \ref{DlambdaPrimeProp}, and the decay of $D_m(\rho)$, for $\rho\geq R_8$
$$
\frac{d}{d\rho} \left( \rho^{3-n} D_m(\rho)\right) \leq -\rho^{3-n}  \int_{S_\rho} (\partial_r u)^2 +2 \rho^{2-n} D_m(\rho)\leq - \rho^{3-n}  \int_{S_\rho} (\partial_r u)^2+\frac{4\bar{\xi} \alpha^2}{ \rho^{2}}.
$$
Integrating, using the co-area formula and the decay of $D_m(\rho)$ gives, for $R\geq R_8$,
$$
\int_{\bar{E}_R} r^{3-n} (\partial_r u)^2 \leq 12\bar{\xi}\alpha^2 R^{-1} \leq 24 \bar{\xi} R^{-n} B(R).
$$
Taken together theses estimates prove the first inequality with $K_8=96\bar{\xi}$.

Next observe that by what we have shown, for $R\geq R_8$,
	$$
\int_{E_R} r^{-1-n}u^2+  r^{1-n}|\nabla_g u|^2 + r^{3-n}(\partial_r u)^2 \leq K_8\alpha^2  R^{-1} .
$$ 
Hence, by Proposition \ref{AsympHomogProp}, $u$ is asymptotically homogeneous of degree $0$ and $a=\mathrm{tr}_\infty^0(u)\in H^1(L(\Sigma))$.
Clearly, $\alpha^2=||a||_{L^2(L(\Sigma))}^2$ and first estimate allows us to apply Proposition \ref{AsympHomogProp} to conclude the proof.  
\end{proof}

\section{General almost $\LL_m$ eigenfunctions}
In this section we show how almost $\LL_m$ eigenfunctions can be transformed into almost $\LL_{m'}$ harmonic functions and also transformed back.  We will also demonstrate a transformation between $\LL_m$ eigenfunctions and eigenfunctions of the operator
$$
\LL_m^+=\Delta_g +\frac{r}{2}\partial_r +\frac{m}{r}\partial_r,
$$
which is associated to the weight
$$
\Psi_m=r^m e^{\frac{r^2}{4}}.
$$
Before proving this we record the following motivating computations:
\begin{align*}
\LL_m r^\mu =-\frac{1}{2}  \mu r^{\mu}  +O( r^{-2+\mu}; \mu, m )\mbox{ and }
\LL_m^+ r^\mu =\frac{1}{2}  \mu r^{\mu}  +O( r^{-2+\mu}; \mu, m ).
\end{align*}
\begin{align*}
\LL_m \Psi_\mu &=
 \frac{1}{2}(\mu+n+m)  \Psi_\mu+O(r^{-2} \Psi_\mu; \mu,m).
\end{align*}
\begin{align*} 
\LL_m^+ \Phi_\mu &= - \frac{1}{2}( \mu +n+m)  \Phi_\mu+O(r^{-2} \Phi_\mu; \mu,m).
\end{align*}
\begin{prop} \label{ChangeOfSpecProp}
There is an $M'=M'(M,m, \mu, n, \Lambda)$ so:  If $u\in C^2(\Sigma)$ satisfies
\begin{enumerate}
\item $\left|\left(\LL_m+\lambda\right) u\right|\leq M r^{-2} \left(|u|+|\nabla_g u|\right),$
then $\hat{u}=r^{2\mu}u$ satisfies $$
\left|\left(\LL_{m-4\mu}+\lambda+\mu\right) \hat{u}\right|\leq M' r^{-2} \left(|\hat{u}|+|\nabla_g \hat{u}|\right) .
$$
\item $
\left|\left(\LL_m+\lambda\right) u\right|\leq M r^{-2} \left(|u|+|\nabla_g u|\right).
$
then, $\hat{u}=\Phi_{\mu} {u}$ satisfies
$$
\left|\left(\LL_{m-2\mu}^++\frac{1}{2}\left(n+m+2\lambda-\mu\right) \right) \hat{u}\right|\leq M' r^{-1} \left(|\hat{u}|+r^{-1}|\nabla_g \hat{u}|\right).
$$
\item $
\left|\left(\LL_m^++\lambda\right) u\right|\leq M r^{-2} \left(|u|+r^{-1}|\nabla_g u|\right),
$
then 
$\hat{u}=\Psi_{\mu}(r) u$ satisfies
$$
\left|\left(\LL_{m-2\mu}+\frac{1}{2}\left(-n-m+2\lambda+\mu\right)\right) \hat{u}\right|\leq M' r^{-2} \left(|\hat{u}|+|\nabla_g \hat{u}|\right).
$$

\end{enumerate} 
\end{prop}
\begin{proof}
If $\hat{u}=r^{2 \mu} u$, then
$\nabla_g \hat{u}=r^{2 \mu} \nabla_g u+2\mu r^{2\mu-1} u \partial_r 
$
and so
$$
r^{2 \mu} |\nabla_g u|\leq |\nabla_g \hat{u}| +\frac{4|\mu|}{r} r^{2\mu} |u|=|\nabla_g \hat{u}| +\frac{4|\mu|}{r} |\hat{u}|.
$$
One computes
\begin{align*}
\LL_{m'} \hat{u} &=r^{2 \mu}\LL_{m'}u +4\mu r^{2\mu-1} \partial_r u +(\LL_{m'} r^{2\mu}) u\\
 &=r^{2 \mu} \LL_{m'+4\mu}u-\mu \hat{u}+\hat{u}O(r^{-2}; \mu, m').
\end{align*}
The previous two computations give
\begin{align*}
\left|\left(\LL_{m-4\mu}+\lambda+\mu\right) \hat{u}\right| &= r^{2\mu}|(\LL_m+\lambda)u|  + |\hat{u}| O(r^{-2}; \mu,m)\\
&\leq M  r^{2\mu-2} \left( |u|+|\nabla_g u|\right) +|\hat{u}| O(r^{-2}; \mu,m)\\
&\leq M r^{-2} |\nabla_g \hat{u}| + |\hat{u}| O(r^{-2}; \mu,m,M)
\end{align*}
This shows (1) for an appropriate $M'$. 

If $\hat{u}=\Phi_{\mu} u$, then $\nabla_g \hat{u}=\Phi_{\mu} \nabla_g u+( -\frac{r}{2} +\mu r^{-1}) \Phi_{\mu} u \partial_r.$
Hence,
$$
\Phi_{\mu} |\nabla_g u|\leq |\nabla_g \hat{u}|+\left(r+2|\mu| r^{-1}\right) \Phi_{\mu} |u|=|\nabla_g \hat{u}|+\left(r+2|\mu| r^{-1}\right) |\hat{u}|.
$$
One computes that
\begin{align*}
\LL^+_{m'} \hat{u}&= \Phi_{\mu} \LL_{m'}^+ u +2 g(\nabla_g \Phi_{\mu}, \nabla_g u) + (\LL_{m'}^+ \Phi_{\mu}) u\\
&=\Phi_{\mu} \LL_{m'}^+ u -r \Phi_{\mu} \partial_r u+2\mu r^{-1} \Phi_{\mu} \partial_r u -\frac{1}{2}(\mu+n+m') \hat{u} +\hat{u} O(r^{-2}; \mu, m')\\
&=\Phi_{\mu} \LL_{m'+2\mu} u-\frac{1}{2}(\mu+n+m') \hat{u} +\hat{u} O(r^{-2}; \mu, m')
\end{align*}
Hence, for a large enough $M'$,
\begin{align*}
&\left|\left(\LL_{m-2\mu}^++\frac{1}{2}(n+m+2\lambda-\mu)\right)\hat{u}\right| \leq  \Phi_{\mu}| (\LL_m+\lambda) u| + |\hat{u}| O(r^{-2}; \mu, m)\\
&\leq  \Phi_{\mu} M r^{-2} \left( |u|+|\nabla_g u|\right) + |\hat{u}| O(r^{-2}; \mu, m)\leq  M r^{-2} \left( |\nabla_g \hat{u}|+r|\hat{u}|\right)\\
& + |\hat{u}| O(r^{-2}; \mu, m,M)\leq M' r^{-1}\left( |\hat{u}|+r^{-1}|\nabla_g \hat{u}|\right)
\end{align*}
which verifies (2).  The verification of (3) is nearly identical.
\end{proof}
The Theorem \ref{MainThm2} follows immediately by applying Theorem \ref{AsymptoticEstThm} to $\hat{u}=r^{-2\lambda} u$ and appealing to Proposition \ref{ChangeOfSpecProp}.  Theorem \ref{AsymptoticEstThm} and Proposition \ref{ChangeOfSpecProp} can also be used to prove the following decay result for almost $\LL_0^+$ eigenfunctions.

\begin{thm}\label{MainThm2Alt}  If $ \hat{u}=\Psi_{n-2\lambda}(r) u\in C^2_{m+2,1}(\Sigma)$ and $u$ satisfies
 $$\left|(\LL_0^+ +\lambda) u\right|\leq M r^{-2}\left(  |u|+r^{-1}|\nabla_g u|\right)$$
	then, there are constants $R_9$ and $K_9$, depending on $u$, so that for $R\geq R_9$
	$$
	 \int_{\bar{E}_R} \left(\hat{u}^2+  r^{2}|\nabla_g \hat{u}|^2 + r^{4}\left(\partial_r \hat{u}\right)^2 \right) r^{-1-n}\leq \frac{K_9}{R^n}\int_{S_R} \hat{u}^2.
	$$ 
	Moreover,  $\hat{u}$ is asymptotically homogeneous of degree $0$ and there is an element $
		\mathrm{tr}_\infty^0 \hat{u}=\hat{a}$ so that $\alpha^2=\lim_{\rho\to \infty} \rho^{1-n} \int_{S_{\rho}} \hat{u}^2=\int_{L(\Sigma)}\hat{a}^2$ and for $R\geq R_9$
		$$
		\int_{\bar{E}_R} \left(\hat{u}^2+r^2\left(|\hat{u}-\hat{A}|^2+ |\nabla_g \hat{u}|^2\right)+ r^{4}\left(\partial_r \hat{u}\right)^2\right) r^{-2-n} \leq \frac{K_9 \alpha^2}{R^2}.
		$$ 
		Here $\hat{A}\in H^1_{loc}(\Sigma)$ is the leading term of $\hat{u}$ and $L(\Sigma)$ is the link of the asymptotic cone.  	
\end{thm}

\section{Self-Shrinkers and Self-expanders}
In this section we use Theorems \ref{MainThm2} (resp. Theorem  \ref{MainThm2Alt}) to prove Theorem \ref{WangThm} (resp. Theorem \ref{ExpanderThm}).   We begin by observing that the ends of asymptotically conical self-shrinkers and self-expanders are  weakly conical ends. 
\begin{lem}\label{WeakConeLem}
Let $\Sigma\subset \Real^{n+1}$ be an asymptotically conical self-shrinker or self-expander.  There is an $R_\Sigma$ so that if $\Sigma'=\Sigma\backslash \bar{B}_{R_\Sigma}$, $g$ is the induced metric from $\Real^{n+1}$ and $r(p)=|\mathbf{x}(p)|$, then $(\Sigma', g, r)$ is a weakly conical end.
\end{lem}
\begin{proof}
Let $C$ be the asymptotic cone of $\Sigma$.  By definition, there is a $K> 0$ a radius $R_\Sigma$, both depending on $\Sigma$, so that for $p\in \Sigma\backslash B_{R_\Sigma}$, 
$$
|A_\Sigma|\leq K r(p)^{-1}.
$$
Increase, if needed, $R_\Sigma$ so $\partial \Sigma \cap B_{R_\Sigma}=\emptyset$ and $\sqrt{n} K R_\Sigma^{-1}+2n K^2 R_{\Sigma}^{-2}<\frac{1}{4}$.  Both the self-shrinker and self-expander equations imply that
for any $p\in \Sigma'=\Sigma\backslash \bar{B}_{R_\Sigma}$, 
$$
|\mathbf{x}\cdot \mathbf{n}|(p)=2|H_{\Sigma}(p)|\leq 2\sqrt{n} |A_\Sigma(p)|\leq 2 \sqrt{n} K r(p)^{-1}<\frac{1}{2}.$$
Hence, $|\mathbf{x}^\top|\geq \frac{1}{2}|\mathbf{x}|$ and so
$$
||\nabla_g r|-1|=\left| \frac{|\mathbf{x}^\top|}{|\mathbf{x}|}-1\right|=\frac{|\mathbf{x}|^2-|\mathbf{x}^\top|^2}{|\mathbf{x}|(|\mathbf{x}^\top|+|\mathbf{x}|)}\leq \frac{2|\mathbf{x}\cdot \mathbf{n}|^2}{3 |\mathbf{x}|^2}\leq 4n K^2 r^{-4}<\frac{1}{2}.
$$
Similarly, as
$$
\nabla_g^2 r^2=\nabla_g^2 |\mathbf{x}|^2= 2g -2\mathbf{x}\cdot \mathbf{n} A_{\Sigma},
$$
$$
|\nabla_g^2 r^2-2g|=2|\mathbf{x}\cdot \mathbf{n}||A_\Sigma|\leq 4 \sqrt{n} K^2 r^{-2} <\frac{1}{2}.
$$
Together these imply that $(\Sigma', g, r)$ is a weakly conical end.
\end{proof}

We will also need certain straightforward estimates on the graph representing the difference between self-similar solutions.  These are proved (by essentially the same argument) for self-shinkers in Lemmas 2.3 and 2.4 of \cite{WaRigid} and for self-expanders in Lemma 5.2 and Corollary 5.4 of \cite{Ding} and so we omit a proof.
\begin{lem}\label{GraphLem}
Suppose $\Sigma_1$ and $\Sigma_2$ are two asymptotically conical self-shrinkers (resp.  self-expanders) with the same asymptotic cone.  There is a radius $R_G$ and a function $u\in C^{2}(\Sigma_1)$, so that the graph of $u$ over $\bar{E}_{R_G}\subset \Sigma_1$ is contained in $\Sigma_2$. 

Moreover, there is a constant $\kappa$ so that, on $\bar{E}_{R_G}$, $u$ satisfies
\begin{equation}
\label{uEst}
r|u|+r^2|\nabla_g u|\leq \kappa,
\end{equation}
and in the self-shrinking case
\begin{equation}
\label{SSEqn}
\left|\left(\LL_0+\frac{1}{2}\right) u\right|\leq \kappa r^{-2} \left(|u|+|\nabla_g u|\right),
\end{equation}
while in the self-expanding case
\begin{equation}
\label{SEEqn}\left|\left(\LL_0^+-\frac{1}{2}\right) u\right|\leq \kappa r^{-2} \left(|u|+r^{-1}|\nabla_g u|\right).
\end{equation}

\end{lem}

We are now ready to prove Theorem \ref{WangThm}:
\begin{proof}(of Theorem \ref{WangThm})
Suppose $\Sigma_1$ and $\Sigma_2$ are two asympotically conical self-shrinkers asymptotic to the same cone.  By Lemma \ref{WeakConeLem} and Lemma \ref{GraphLem}, if $\tilde{R}=\max\set{ R_{\Sigma_1'}, R_G}$, then $\Sigma_1'=\Sigma_1\backslash \bar{B}_{\tilde{R}}$ is a weakly conical end and there is an element $u\in C^2(\Sigma_1')$ satisfying \eqref{uEst} and \eqref{SSEqn}.  Hence, $u$ satisfies the hypotheses of Theorem \ref{MainThm2} with $\lambda=\frac{1}{2}$ and $
\lim_{\rho\to \infty} \rho^{-2-n} \int_{S_\rho} u^2=0.$
  It follows from Theorem \ref{MainThm2} that there is a $R_U\geq \tilde{R}$ so that $u=0$ on $\bar{E}_{R_U}\subset \Sigma_1$ and so $\Sigma_1\backslash B_{R_U}=\Sigma_2\backslash B_{R_U}$.
\end{proof}

We can also prove Theorem \ref{ExpanderThm}.
\begin{proof}(of Theorem \ref{ExpanderThm})
Suppose $\Sigma_1$ and $\Sigma_2$ are two connected asympotically conical self-expanders.  The decay of the Hausdorff distance ensures they have  the same asymptotic cone. By Lemma \ref{WeakConeLem} and Lemma \ref{GraphLem}, if $\tilde{R}=\max\set{ R_{\Sigma_1'}, R_G}$, then $\Sigma_1'=\Sigma_1\backslash \bar{B}_{\tilde{R}}$ is a weakly conical end and there is an element $u\in C^2(\Sigma_1')$ satisfying \eqref{uEst} and \eqref{SEEqn}.
As such, $u$ satisfies the hypotheses of Proposition \ref{StrongDecayThm} and, hence, those of Theorem \ref{MainThm2Alt}.

Geometric considerations give a constant $K$, so that for any $\rho\geq \tilde{R}$, 
$$
\int_{S_\rho} u^2 \leq K \rho^{n-1}\dist_H(\Sigma_1\cap \partial B_\rho,\Sigma_2\cap \partial B_\rho)^2\leq \rho^{-n-3} \Phi_0^2(\rho)o(1).
$$
Hence, if $\hat{u}=\Psi_{n+1} u$, then $
\lim_{\rho\to \infty} \rho^{1-n} \int_{S_\rho} \hat{u}^2 =0.$
  It follows from Theorem \ref{MainThm2Alt} that there is a $R_U\geq \tilde{R}$, so that $u=0$ on $\bar{E}_{R_U}\subset \Sigma_1$ and so $\Sigma_1\backslash B_{R_U}=\Sigma_2\backslash B_{R_U}$.\end{proof}

\section{Decay estimates}\label{DecayEst}
In this section we use an energy argument to show that if the boundary $L^2$ norm of an almost $\LL_0^+$ eigenfunction grows below a critical threshold, then it is in the appropriate weighted spaces needed to apply Theorem \ref{MainThm2Alt}.

\begin{thm}\label{StrongDecayThm}
  If $u\in C^2(\bar{E}_R)$ satisfies
$$\left|\left(\LL^+_0 +\lambda\right)u\right|\leq M r^{-1} \left(  |u|+|\nabla_g u|\right) \mbox{ and } B(\rho)=o(\rho^{-4\lambda+n-1}), \rho\to \infty,$$ then $\Psi_0 u\in C^2_{m',1}(\bar{E}_R)$ for any $m'$. 
\end{thm}

Given a function $u\in C^2(\bar{E}_R)$ and $r_2>r_1$ let
$$
\check{D}_m (u,r_1,r_2) =\int_{\bar{A}_{r_2, r_1}} |\nabla_g u|^2 \Psi_m.
$$
\begin{lem}\label{GradControlsL2lem}
There is an $R_P'=R_P'(n,\Lambda,m)$ so that for any $s>t\geq R_P'$ and and $\phi \in C^2(\bar{A}_{s, t})$,
\begin{equation*}
\int_{\bar{A}_{s,t}} r^{-1} \phi^2 \Psi_m  \leq\frac{32}{t^3} \check{D}_m(\phi, t,s)+\frac{8}{ s^2}\Psi_m(s) \int_{S_{s}} |\nabla_g r| \phi^2
\end{equation*}
\end{lem}
\begin{proof}
Consider the vector field $
\mathbf{V}(p)=\frac{1}{r(p)^2} \phi^2(p) \Psi_m(p) \partial_r$
 for which
$$
\Div_g \mathbf{V}= \frac{n+m-3}{r^3} \phi^2 \Psi_m + \frac{2}{r^2} \phi \partial_r \phi \Psi_m + \frac{1}{2r} \phi^2 \Psi_m + \phi^2 \Psi_m O(r^{-5}; m).
$$
When $R_P'\geq 4\sqrt{n+|m|}$ is large enough,  the absorbing inequality yields
$$
\Div_g \mathbf{V}\geq -4r^{-3}|\nabla_g \phi|^2 \Psi_m+ \frac{1}{8 }r^{-1}\phi^2 \Psi_m. 
$$
Integrating, 
gives
$$
\frac{1}{8}\int_{\bar{A}_{s,t}} \frac{\phi^2}{r} \Psi_m +\frac{\Psi_m(t)}{t^2}\int_{S_{t}}|\nabla_g r| \phi^2-\frac{\Psi_m(s)}{s^2}\int_{S_{s}}|\nabla_g r| \phi^2\leq 4 \int_{\bar{A}_{s,t}} \frac{|\nabla_g \phi|^2 }{r^3} \Psi_m,
$$
which proves the result.
\end{proof}

\begin{lem}		\label{IntegrationCor}
Given an $m\in \Real$ and $M\geq 0$, there is an $R_{10}=R_{10}(n, \Lambda, m,M)$ and $K_{10}=K_{10}(M)>0$ so that if $R\geq R_{10}$ and $u\in C^2(\bar{E}_R)$ satisfies
\begin{equation}
|\LL^+_m u|\leq M r^{-1} (|u|+|\nabla_g u|),
\end{equation}
then, for all $r_2>r_1\geq R$
$$
\Psi_m(r_1)F(r_1)-\Psi_m(r_2)F(r_2) \geq-\frac{K_{10}}{r_2^2}\Psi_m(r_2)B(r_2).
$$
\end{lem}
\begin{proof}

The divergence theorem gives,
$$
\Psi_m(r_2)\int_{S_{r_2}} u (\mathbf{N}\cdot  u) - \Psi_m(r_1) \int_{S_{r_1}} u  (\mathbf{N}\cdot u)=\check{D}_{m}(u,r_1,r_2) +\int_{\bar{A}_{r_2,r_1}} u \LL^+_m u \Psi_m.
$$
Hence, as long as $R_{10}\geq R_P'+6$ is large enough, Lemma \ref{GradControlsL2lem} implies that
\begin{align*}
\Psi_m(r_1)F(r_1)&- \Psi_m(r_2) F(r_2)\geq \check{D}_{m}(u,r_1,r_2) -2M \int_{\bar{A}_{r_2,r_1}} r^{-1} (u^2+|\nabla_g u|^2) \Psi_m
\\
&\geq (1-\frac{2M}{r_1})\check{D}_m(u,r_1,r_2)-2M\int_{A_{r_2,r_1}} r^{-1} u^2 \Psi_m\\
&\geq (1-\frac{2M+1}{r_1})\check{D}_m(u,r_1,r_2)-\frac{16M}{r_2^2} \Psi_m(r_2)B(r_2).
\end{align*}
As $\check{D}(u,r_1,r_2)\geq 0$, it suffices to ensure $R_{10}\geq 2M+1$ and take $K_{10}=16M$.
\end{proof}

\begin{proof}(of Theorem \ref{StrongDecayThm})
 Observe that if $\hat{u}=r^{2\lambda} u$, then $\hat{u}$ satisfies, for some $\hat{M}$, 
	$$
	\left|\LL_0^+\hat{u}\right|\leq \hat{M} r^{-1}\left(|\hat{u}|+|\nabla_g \hat{u}|\right).
	$$
As such, it suffices to prove the theorem for $\lambda=0$.	Moreover, for any $m\in \Real$,
$$
|\LL_{m}^+ u|=|\LL_{0}^+ u+\frac{m}{r} \partial_r u|\leq (\hat{M}+|m|) r^{-1} \left( |u|+|\nabla_g u|\right).
$$

Fix an $m\in \Real$ and let $R_{10}=R_{10}\left( n,\Lambda, m, \hat{M}+|m|\right)$ and $K_{10}=K_{10}(\hat{M}+|m|)$ be given by Lemma \ref{IntegrationCor}.  For each $t> r_1\geq R_{10}$, Lemma \ref{LHatPrimeLem} and Lemma \ref{IntegrationCor} imply
\begin{align*}
B'(t) &= \frac{n-1}{t} B(t)- 2F(t) +O(t^{-3}) B(t)\\ 
&=\frac{n-1}{t} B(t) - 2\Phi_{-m}(t)\Psi_m(t)F(t)+ 2\Phi_{-m}(t)\Psi_m(r_1)F(r_1)\\
& -2\Phi_{-m}(t) \Psi_m(r_1)F(r_1)+O(t^{-3}) B(t)\\
&\geq \frac{n-1 -2K_{10}t^{-1}}{t} B(t) + C(u,r_1)\Phi_{-m}(t). 
\end{align*}
Here, $C(u,r_1)$ depends on $u$ and $\nabla_g u$ on $S_{r_1}$, but is independent of $t$.
Hence, 
$$
\left(e^{-\frac{2K_{10}}{t}}t^{1-n} B(t)\right))'\geq  C(u,r_1)e^{-\frac{2K_{10}}{t}}  t^{1-n}\Phi_{-m}(t)\geq e^{-\frac{2K_{10}}{r_1}} C(u,r_1) \Phi_{1-n-m}(t).
$$
Integrating from $\rho\geq r_{1}$ to $\infty$ and using the hypothesis on the asymptotic growth of $B(t)$ and the fact that $\lim_{t\to \infty} e^{-\frac{2K_{10}}{t}}=1$, implies that
$$
-e^{-\frac{2K_{10}}{r_1}} C(u,r_1)\int_\rho^\infty \Phi_{1-n-m} (t)dt \geq e^{-\frac{2K_{10}}{r_1}}\rho^{1-n} B(\rho)
$$
Computing as in \eqref{PhimIntEqn}, for $\rho$ sufficiently large (depending on $m$ and $n$),
$$
-4C(u,r_1) e^{\frac{2K_{10}}{r_1}} \Phi_{-m-1}(\rho)\geq B(\rho).
$$ 

As $m$ can be taken to be arbitrary, for the given $m'$, by setting $m=-4-m'$ the co-area formula immediately implies there is a $C>0$, depending on $u$,  so that,
$$
\int_{\bar{E}_R} (1+r^2) u^2 \Psi_{m'} <\frac{C}{R^2}<\infty.
$$
Let $\eta\in C^2(\Real)$ satisfy, $0\leq \eta\leq 1$,  $\eta=1$ on $[2, 3]$ and $\mathrm{spt}(\eta)\subset [1,4]$ and set $K=\sup_{\Real}( |\eta|+ |\eta'|+|\eta''|)$.  
For $\rho> 10 R$, let $\eta_\rho\in C^2(\Sigma)$ be defined by $\eta_\rho(p)=\eta\left( \rho^{-1}r(p)\right)$.
Integrating by parts gives for $\rho>10 R$, 
\begin{align*}
\int_{\Sigma}( \LL^+_{m'} \eta_\rho ) u^2 \Psi_{m'} &= \int_{\Sigma} \eta_\rho \LL^+_{m'}(u^2)\Psi_{m'} = 2\int_{\Sigma} \eta_\rho (|\nabla_g u|^2 +u \LL^+_{m'} u) \Psi_{m'}\\
 &\geq \int_{\bar{A}_{2\rho, 3\rho}} |\nabla_g u|^2 \Psi_{m'}- 2(\hat{M}+|m'|+1) \int_{\bar{A}_{\rho, 4\rho}} u^2 \Psi_{m'}.
\end{align*}
Hence, for $K'=K+2(\hat{M}+|m'|+1)$
\begin{align*}
 \int_{\bar{A}_{2\rho, 3\rho}} |\nabla_g u|^2 \Psi_{m'}\leq K' \int_{\bar{A}_{\rho,4\rho}} u^2 \Psi_{m'}\leq CK' \rho^{-2}.
\end{align*}
It follows from this that
$$
\int_{\bar{E}_{R}} |\nabla_g u|^2 \Psi_{m'}<\infty,
$$
and so if $\hat{u}=\Psi_0 u$, then $\Psi_0 u\in C^{2}_{m',1}(\bar{E}_R)$ for any $m'$.  Indeed,
$$
\int_{\bar{E}_R}\left(|\nabla_g \hat{u}|^2+ \hat{u}^2\right)\Phi_{m'}\leq2 \int_{E_R} \left(  u^2 + |\nabla_g u|^2 +\frac{r^2}{4} u^2 \right) \Psi_{m'}<\infty.
$$
\end{proof}

\appendix
\section{Asymptotically homogenous functions and traces at $\infty$}
\label{AsymptoticSec}
In this section we define asymptotically homogeneous functions on weakly conical ends.  As part of this, we will need to show that any weakly conical end $(\Sigma, g, r)$ is asymptotic (in a natural way) to a $C^0$-Riemannian cone. 
 
We begin by letting
$$L(\Sigma)=S_{R_L}=S_{R_\Sigma+1}$$
be the \emph{link} of $\Sigma$.
For any $\tau\geq 1$ let 
$$\Pi_\tau: \Sigma \to E_{\tau R_\Sigma}$$
be the time $\ln\tau$ flow of $\mathbf{X}$ on $\Sigma$. As $\mathbf{X} \cdot r=r$, $\Pi_\tau (S_\rho)=S_{\tau\rho}$ and $\Pi_\tau(E_\rho)=E_{\tau \rho}$.
As such, the restriction of $\Pi_{\tau}$ to $L(\Sigma)$ gives a diffeomorphism $\pi_{\tau}: L(\Sigma)\to S_{\tau}$.  Denote by $g(\tau)$ the $C^1$-Riemannian metric induced by $g$ on $S_\tau$ and define
$$
g_L(\tau)= \pi_\tau^*\left(r^{-2} g(\tau)\right) =\frac{\pi_{\tau}^{*}g(\tau)}{\tau^2 R_L^2}.
$$
be a metric on $L(\Sigma)$.
We claim that
$$
\lim_{\tau\to \infty} g_L(\tau)=g_L \mbox{ in $C^0(L(\Sigma))$}
$$
where $g_L$ is a $C^0$-Riemannian metric on $L(\Sigma)$.
To see this observe that
\begin{align*}
\frac{d}{d\tau}g_L(\tau)&= \pi_{\tau}^*\left(r^{-2}\frac{d}{dh}|_{h=0} \Pi_{1+\frac{h}{\tau}}^* g(\tau+h))\right)-\frac{2g_L(\tau)}{\tau}\\
&=\pi_{\tau}^*\left( 2 r^{-2} \tau^{-1}{\mathbf{X}}\cdot \mathbf{N} A_{S_{\tau R_L}}\right)-\frac{2g_L(\tau)}{\tau}\\
&=\pi_\tau^*\left( 2 r^{-2}\tau^{-1}\frac{r}{|\nabla_g r|} \left(\frac{1}{r} g(\tau)+r^{-3} \omega(\tau))\right)\right)-\frac{2g_L(\tau)}{\tau}\\
&=2\tau^{-5} \frac{1}{|\nabla_g r|} \pi_\tau^* \omega(\tau) +2\tau^{-1} \frac{1-|\nabla_g r|}{|\nabla_g r|} g_L(\tau)
\end{align*}
Here, $\omega(\tau)$ is a symmetric $(0,2)$ tensor that satisfies
$$
-\lambda  g(\tau)\leq \omega(\tau) \leq \lambda g(\tau).
$$
for a constant $\lambda$ depending on $\Lambda$.  Up to increasing $\lambda$, this implies
$$
-\lambda \tau^{-3} g_L(\tau)\leq \frac{d}{d\tau}g_L(\tau)\leq \lambda \tau^{-3} g_L(\tau). 
$$
Hence, for $\tau_2\geq \tau_1\geq  1$,
$$
e^{-\frac{1}{2}\lambda \tau^{-2}_1} g_L(\tau_1) \leq   g(\tau_2)\leq e^{\frac{1}{2}\lambda \tau^{-2}_1} g_L(\tau_1)
$$
which implies the limit exists in $C^0$ topology and it is positive definite.
Moreover,
$$
e^{-\frac{1}{2} \lambda \tau^{-2}} g_L\leq g_L(\tau)\leq e^{\frac{1}{2} \lambda \tau^{-2}} g_L.
$$
A consequence of this is that, if $g_\tau=  \tau^{-2}\Pi^*_{\tau} g$, then
$$
\lim_{\tau\to \infty} g_\tau= dr^2 +r^2 g_L =g_C\mbox{ in $C^0_{loc}(\Sigma)$}.
$$
That is, $(\Sigma, g_C)$ is a $C^0$-Riemannian cone with link $(L(\Sigma), g_L)$ called the \emph{asymptotic cone  of $(\Sigma,g, r)$}.  In particular, for each $\epsilon>0$, there is an $R_\epsilon$ so that on $E_{R_\epsilon}$,
$$
(1-\epsilon) g_{C}\leq g \leq (1+\epsilon)g_{C}
\mbox{ and }(1-\epsilon)^{n/2} d\mu_{C}\leq d\mu_g \leq (1+\epsilon)^{n/2} d\mu_{C}.
$$
Here $d\mu_g$ is the density associated to $g$ and $d\mu_{C}$ the one associated to $g$.

There is a well defined $L^2$ and $H^1$ norm associated to any $C^0$-Riemannian metric.  For instance, on $(L, g_L)$ there are the norms.
$$
||f||_{L^2(L)}=\int_{L} f^2 d\mu_L \mbox{ and } ||f||_{H^1(L)}=\int_{L} |\nabla_{g_L} f|^2 d\mu_L +||f||_{L^2(L)}.
$$
Furthermore, it follows from the above computations that for any compact set $K\subset \Sigma$ there is a $\tau_0=\tau_0(K)$ so that for any $f\in C^1(K)$ and any $\tau\geq \tau_0$, 
$$
\frac{1}{2}\int_{K} f^2 d\mu_{C}\leq \int_{K} f^2d{\mu}_{g_\tau}=\int_{K} f^2\Pi_\tau^*(r^{-n} d{\mu}_{g})\leq 2\int_{K} f^2 d\mu_{C}
$$
and
$$
\frac{1}{2}\int_{K} |\nabla_{C} f|^2 d\mu_{C} \leq \int_{K} |\nabla_{g_\tau} f|^2 d{\mu}_{g_\tau}=2\int_{K}  |\nabla_{C} f|^2  d\mu_{g_C}
$$

An element of $F\in L^2_{loc}(\Sigma; d\mu_{C})$ is \emph{homogeneous of degree $d$} if, for all $\tau\geq 1$,  $\Pi_\tau^*\left(r^{-d}F\right) =r^{-d}F$.  If $F$ is homogeneous of degree $d$, then there is a unique element, $f\in L^2(L(\Sigma); g_L)$, given by restricting $r^{-d}F$ to $L(\Sigma)$.  Denote this by
$$
f=\mathrm{tr}^d(F).
$$  
Observe, that if $F\in H^1_{loc}(\Sigma; g_C)$, then $\mathrm{tr}^d (F)\in H^1(L(\Sigma);g_L)$.

More generally, $G\in L^2_{loc}(\Sigma; g)$ is \emph{asymptotically homogeneous of degree $d$} if 
$$
\lim_{\tau\to \infty} \Pi^{*}_\tau\left( r^{-d}G\right)\to r^{-d} F \mbox{ in $L^2_{loc}(\Sigma; d\mu_C)$}
$$
where $G$ is homogeneous of degree $d$.  Call $F$, the \emph{leading term} of $G$.  If $G$ is asymptotically homogeneous of degree $d$ with leading term $F$, then let
$$
\mathrm{tr}_\infty^d(G)=\mathrm{tr}^d(F)
$$
be the \emph{trace at infinity} of $G$.

\begin{prop} \label{AsympHomogProp}
	If $G\in C^1(\Sigma)$ satisfies
	$$
	\int_{E_{R}} r^{-n} |\nabla_g G|^2 + r^{2-n} (\partial_r G)^2 d\mu_{g }\leq  \alpha^2 R^{-2}
	$$
	for all $R\geq R_H$, then $G$ is asymptotically homogeneous of degree $0$. Moreover, if $F$ is the leading term of $G$, then $F\in H^1_{loc}(\Sigma)$ and
	$$
	\int_{E_R} r^{-n}|F-G|^2 d\mu_{g} \leq 16 \alpha^2 R^{-2}.
	$$
\end{prop}
\begin{proof}
  Let $ G_\tau=\Pi_\tau^* G$ and observe that
 $$
 \frac{d}{d\tau} G_\tau(p)=\left.\frac{d}{dh}\right|_{h=0} G_{\tau+h}(p)=\left.\frac{d}{dh}\right|_{h=0} G(\Pi_{1+\frac{h}{\tau}}(\Pi_\tau(p))) = \frac{1}{\tau}\left( \mathbf{X}\cdot G\right)(\Pi_\tau(p)).
 $$
 Hence, for $K=\bar{A}_{\rho_2,\rho_1}$ and $R'\geq R\geq 1$, the Cauchy-Schwarz inequality and Fubini's theorem give
 \begin{align*}
  \int_{K} &| G_{R'}-G_{R}|^2 d\mu_C = \int_K  \left| \int_{R}^{ R'}  \frac{d}{ds} G_s \; ds \right|^2 d\mu_C=\int_K  \left| \int_{R}^{ R'}  \frac{1}{s}  \Pi_s^*(\mathbf{X}\cdot G) ds \right|^2 d\mu_C\\
  &\leq  \int_{R}^{ R'} \frac{1}{s^2} ds \int_{R}^{ R'} \int_K  \left(\Pi_s^*(\mathbf{X}\cdot G)\right)^2 d\mu_C ds\leq\frac{1}{R} \int_{R}^{R'} \int_K  \Pi_s^*(\mathbf{X}\cdot G)^2d\mu_C ds.
\end{align*}
For $R\geq \tau_0(K)$, if $s\geq R$, then  $d\mu_C \leq 2 \Pi_s^*\left(r^{-n} d\mu_g\right).$ Hence, for such $R$,
\begin{align*}
\int_{K} |  G_{ R'}&-G_{R}|^2 d\mu_C \leq\frac{2}{R} \int_{R}^{ R'} \int_K  \left( \Pi_s^*(\mathbf{X}\cdot G)\right)^2 \Pi_s^*(r^{-n} d\mu_g) ds\\
&= \frac{2}{R} \int_{R}^{R'} \int_{\Pi_s(K)} r^{-n} (\mathbf{X}\cdot G)^2 d\mu_g ds. 
\end{align*}
The hypotheses ensure that, for $R$ large enough,
 \begin{align*}
 \int_{K} |&  G_{ R'}-G_{R}|^2 d\mu_C \leq \frac{8}{R}\int_{R}^{R'}  \int_{\bar{E}_{\rho_1 s}} r^{2-n}(\partial_r G)^2 d\mu_g ds \\
 &\leq\frac{8 \alpha^2}{\rho_1^2 R}\int_{R}^{R'} \frac{1}{s^2} ds \leq \frac{8 \alpha2}{\rho_1^2 R^2}
 \end{align*}
 
As the right hand side decays in $R$ and is independent of $R'$, it follows that 
$$
\lim_{R\to \infty} G_R =F_K \mbox{ in $L^2(K; d\mu_C)$}
$$
for a unique $F_K\in L^2(K)$. 
Notice that any compact $K'\subset \Sigma$ satisfies $K'\subset K$ for some $K=\bar{A}_{\rho_2, \rho_1}$ and hence this result shows that there is a $F\in L^2_{loc}(\Sigma)$ so that
$$
\lim_{R\to \infty} G_R = F\mbox{ in $L^2_{loc}(\Sigma)$}.
$$
For any $\tau\geq 1$, 
$$F_\tau=\Pi_\tau^* F=\Pi_\tau^*\lim_{R\to \infty} G_R=\lim_{R\to \infty} G_{\tau R}=F
$$
and so $F$ is homogeneous of degree $0$ and hence $G$ is asymptotically homogeneous of degree $0$ with leading term $F$.
By taking $R'\to \infty$ in the above estimate and using that $F$ is homogeneous of degree $0$,
$$
\frac{1}{2}\int_{\bar{A}_{R\rho_2,R\rho_1}}r^{-n} |G -F|^2 d\mu_{g}\leq \int_{\bar{A}_{\rho_2,\rho_1}} |G_R -F|^2 d\mu_{C} \leq  \frac{8\alpha^2}{\rho_1^2 R^2}.
$$
By letting $\rho_2\to \infty$, the dominated convergence theorem gives the claimed estimate.

Finally, observe that for any $  K= \bar{A}_{\rho_2, \rho_1}$ and $R\geq \tau_0(K)$  
\begin{align*}
\int_{K} |\nabla_{C} G_R| ^2 d \mu_{g_C} &\leq 2  \int_{K} \Pi_R^*\left( r^2 |\nabla_{g} G|^2 \right) \Pi^*_R( r^{-n} d\mu_{g})=2 \int_{\Pi_R(K)} r^{2-n}|\nabla_g G|^2 d\mu_g\\
&\leq 2 \rho_2^2 R\int_{\bar{A}_{R\rho_2, R\rho_1}} r^{-n}|\nabla_g G|^2 d\mu_g \leq 2 \rho_2^2 \alpha^2.
\end{align*}
In particular, for any compact subset $K'\subset \Sigma$,  $\limsup_{R\to \infty} ||G_R||_{H^1(K')}<\infty$ and so $F\in H^1_{loc}(\Sigma)$.
\end{proof}
\begin{acknowledgement*}
The author wishes to thank L. Wang for carefully reading this article and for providing many useful comments.
\end{acknowledgement*}

\end{document}